\documentclass[11pt]{amsart}
\usepackage{geometry}                
\usepackage[parfill]{parskip}    
\usepackage{amssymb,  amsmath, mathrsfs}
\usepackage{url}
\usepackage{caption}	
\usepackage{subcaption}
\usepackage{xcolor}
\usepackage{comment}
\usepackage[all,cmtip]{xy}
\usepackage{graphicx}
\usepackage{xypic}
\usepackage{tikz}
\usepackage{bm}
\usepackage{float}
\usepackage{mathtools}

\numberwithin{equation}{section}

\newtheorem{theorem}{Theorem}[section]
\newtheorem{lemma}[theorem]{Lemma}
\newtheorem{proposition}[theorem]{Proposition}

\newtheorem{corollary}[theorem]{Corollary}

\theoremstyle{definition}
\newtheorem{definition}[theorem]{Definition}
\newtheorem{remark}[theorem]{Remark}
\newtheorem{example}[theorem]{Example}

\def\deg{\textup{deg}}
\def\L{\mathcal{L}}
\def\H{\mathbb{H}}

\newcommand{\bH}{\mathcal{H}}
\newcommand{\C}{\mathcal{C}}

\newcommand{\p}{\mathbf{p}}

\newcommand{\Z}{\mathbb{Z}}
\newcommand{\F}{\mathbb{F}}

\newcommand{\A}{\mathfrak{A}}

\newcommand{\mfs}{\mathfrak{s}}
\newcommand{\mft}{\mathfrak{t}}

\newcommand{\Sp}{S^{3}_{p_{1}, \cdots, p_{n}}(\mathcal{L})}

\newcommand{\Co}{\#_{i=1}^{n} B_{p_{i}}}

\newcommand{\os}{Ozsv\'ath and Szab\'o}
\newcommand{\spinc}{\text{Spin}^c}

\newcommand{\relspinc}{\underline{\mathrm{Spin}^c}}

\newcommand{\sq}{\square}

\newcommand{\CW}{\mathrm{CW}}

\begin{document}

\title[Surgery on links and the $d$-invariant]{Surgery on links of linking number zero and the Heegaard Floer $d$-invariant}

\author{Eugene\ \textsc{Gorsky}}
\address{Department of Mathematics\\University of California, Davis \\ One Shields Avenue  \\ Davis, CA 95616 \\ USA}

\address{International Laboratory of Representation Theory\\ and Mathematical Physics\\ National Research University Higher School of Economics\\ Usacheva 6, Moscow, Russia}

\email{egorskiy@math.ucdavis.edu}

\author{Beibei \textsc{Liu}}
\address{Max Planck Institute for Mathematics\\Vivatsgasse 7, 53111 Bonn, Germany}
\email{bbliumath@gmail.com}

\author{Allison H.\ \textsc{Moore}}
\address{Department of Mathematics \& Applied Mathematics\\Virginia Commonwealth University \\ 1015 Floyd Avenue, Box 842014\\ Richmond, VA 23284-2014 \\ USA}
\email{moorea14@vcu.edu}



\begin{abstract}
We study Heegaard Floer homology and various related invariants (such as the $h$-function) for two-component L--space links with linking number zero.  For such links, we explicitly describe the relationship between the $h$-function, the Sato-Levine invariant and the Casson invariant. We give a formula for the Heegaard Floer $d$-invariants of integral surgeries on two-component L--space links of linking number zero in terms of the $h$-function, generalizing a formula of Ni and Wu.
As a consequence, for such links with unknotted components, we characterize L--space surgery slopes  in terms of the $\nu^{+}$-invariants of the knots obtained from blowing down the 
components. 

We give a proof of a skein inequality for the $d$-invariants of $+1$ surgeries along linking number zero links that differ by a crossing change.
We also describe bounds on the smooth four-genus of links in terms of the $h$-function, expanding on previous work of the second author, and use these bounds to calculate the four-genus in several examples of links.
\end{abstract}

\maketitle

\section{Introduction}
\label{sec:intro}

Given a closed, oriented three-manifold $Y$ equipped with a $\spinc$ structure, the Heegaard Floer homology of $Y$ is an extensive package of three-manifold invariants defined by \os~\cite{OS:three}. One particularly useful piece of this package is the $d$-invariant, or \emph{correction term}. For a rational homology sphere $Y$ with $\spinc$ structure $\mft$, the $d$-invariant $d(Y, \mft)$ takes the form of a rational number defined to be the maximal degree of any non-torsion class in the module $HF^-(Y, \mft)$. For more general manifolds, the $d$-invariant is similarly defined (see section \ref{subsec:standard}). The $d$-invariants are known to agree with the analogous invariants in monopole Floer homology (see Remark \ref{monopole}). The terminology `correction term' reflects that the Euler characteristic of the reduced version of Heegaard Floer homology is equivalent to the Casson invariant, once it is corrected by the $d$-invariant \cite{OS:Absolutely}. 
The $d$-invariants have many important applications, for example, to concordance \cite{ManOwens, Peters}, Dehn surgery \cite{NiWu, Doig} and the Heegaard Floer theoretic proofs of Donaldson's theorem and the Thom conjecture \cite{OS:Absolutely}, to name a few. 

From the viewpoint of Heegaard Floer homology, \emph{L--spaces} are the simplest three manifolds.
A rational homology sphere is an L--space if the order of its first singular homology agrees with the free rank of its Heegaard Floer homology. A recent conjecture of Boyer, Gordon and Watson \cite{BGW,HRRW,HRW,Sarah2} describes L--spaces in terms of the fundamental group, and it has been confirmed for many families of 3-manifolds.
A link is an  L--space link if all sufficiently large surgeries on all of its components are L--spaces. 

Given a knot or link in a 3-manifold, one can define its Heegaard Floer homology as well. 
The subcomplexes of the link Floer complex are closely related to the Heegaard Floer complexes of various Dehn surgeries along the link. In the case of knots in the three-sphere, this relationship is well understood by now and, in particular, the following questions have clear and very explicit answers:
\begin{itemize}
\item The formulation of a ``mapping cone'' complex representing the Heegaard Floer complex of an arbitrary rational surgery \cite{OS:Rational}; 
\item  An explicit formula for the $d$-invariants of rational surgeries \cite{NiWu}; 
\item  A classification of surgery slopes giving L--spaces \cite[Proposition 9.6]{OS:Rational}.
\end{itemize}

In this article, we expand the existing Heegaard Floer ``infrastructure" for knots in the three-sphere to the case of links. The work of Manolescu and Ozsv\'ath in \cite{ManOzs} generalizes the ``mapping cone" formula to arbitrary links. For two-component L--space links, their description was made more explicit by Y. Liu \cite{LiuY2} and can be used for computer computations. Both \cite{ManOzs} and \cite{LiuY2} start from an infinitely generated complex and then use a delicate truncation procedure to reduce it to a finitely generated, but rather complicated complex. On the one hand, it is possible to use the work of \cite{ManOzs,LiuY2} to compute the $d$-invariant for a single surgery on a link or to determine if it yields an L--space. On the other hand, to the best of authors' knowledge, it is extremely hard to write a general formula for  $d$-invariants of integral surgeries along links, although such formulas exist for knots in $S^3$ \cite{NiWu} and knots in $L(3, 1)$ \cite{LMV}. 

In general, the characterization of integral or rational L--space surgery slopes for multi-component links is not well-understood. 
The first author and N\'emethi have shown that the set of L--space surgery slopes is bounded from below for most two-component algebraic links and determined this set for integral surgery along torus links \cite{GN:set, GN:plane}. 
Recently, Sarah Rasmussen \cite{Sarah} has shown that certain torus links, satellites by algebraic links, and iterated satellites by torus links have fractal-like regions of rational  L--space surgery slopes.

Nevertheless, in this article we show that the situation simplifies dramatically if the linking number between the link components vanishes. We show that both the surgery formula of \cite{ManOzs} and the truncation procedure lead to explicit complexes similar to the knot case. We illustrate the truncated complexes by pictures that are easy to analyze. 
They are closely related to the lattice homology introduced by N\'emethi \cite{Nemethi,GN}, and best described in terms of the $H$-function $H_{\L}(\bm{s})$, which is a link  invariant  defined over some lattice $\H(\L)$ (see Definition \ref{Hfunction}, see also  \cite{GN}). Note that for a knot $K$, our $H$-function $H_{K}(s)$ agrees with the invariant $V^+_{s}$ of Ni and Wu \cite{NiWu} (see also Rasmussen's \emph{local $h$-invariant} \cite{Ras:Thesis}). For 2-component links $\L$ with vanishing linking number, we define:
\[
	h_{\L}(\bm{s})=H_{\L}(\bm{s})-H_{O}(\bm{s})
\]
where $\bm{s}\in \Z^{2}$ and $H_{O}(\bm{s})$ is the $H$-function of the 2-component unlink. 

Let $S^3_\p(\L)$ denote the $\p=(p_1, \dots, p_n)$ framed integral surgery along an oriented $n$-component link $\L$ in the three-sphere with vanishing pairwise linking number where $p_{i}\neq 0$ for any $i$. We will identify the set of $\spinc$-structures on $S^3_\p(\L)$  with $\Z_{p_1}\times \ldots \times \Z_{p_n}$.
The following result generalizes \cite[Proposition 1.6]{NiWu} and \cite[Theorem 6.1]{OwensStrle}.

\begin{theorem}
\label{thm:generalizedniwu}
The $d$-invariants of integral surgeries on a two-component L--space link with linking number zero can be computed as follows:

(a) If $p_1,p_2<0$ then 
\[
d(S^3_{\p}(\L),(i_1,i_2))= d(L(p_1, 1), i_1) + d(L(p_2, 1), i_2).
\]
(b) If $p_1,p_2>0$ then 
\[
d(S^3_{\p}(\L),(i_1,i_2))=d(L(p_1, 1), i_1) + d(L(p_2, 1), i_2)-2\max\{ h(s_{\pm  \pm}(i_1, i_2)) \},
\]
where $s_{\pm \pm}(i_1,i_2)=(s_{\pm}^{(1)},s_{\pm}^{(2)})$ are four lattice points in $\spinc$-structure $(i_1, i_2)$ which are closest to the origin in each quadrant (see section \ref{d-grading}).

(c) If $p_1>0$ and $p_2<0$ then 
\[
d(S^3_{\p}(\L),(i_1,i_2))=d(S^{3}_{p_1}(L_1), i_1)+ d(L(p_2, 1), i_2).
\]
\end{theorem}

\begin{remark}
If $\L$ is a link with vanishing linking number then all $d$-invariants of all surgeries are concordance invariants. 
\end{remark}

When $p_1=p_2=1$ then $S^3_\p(\L)$ is a homology sphere, and so $i_1,i_2=0$. Moreover $d(L(p_1,1),i_1)=d(L(p_2, 1), i_2)=0$ and $s_{\pm \pm}(0, 0)=(0,0)$, hence
\[
	d(S^{3}_{1, 1}(\L))=-2h(0, 0).
\]
This is analogous to the more familiar equality for knots, $d(S^3_1(K)) = -2V^+_0(K)$, where $V_0(K)$ is the non-negative integer-valued invariant of \cite{NiWu}, originally introduced by Rasmussen as the $h$-invariant $h_{0}(K)$ \cite{Ras:Thesis}.

 As another special case, we consider nontrivial linking number zero L--space links $\L=L_1\cup L_2$ with unknotted components. 
Let $L'_{i}$ $(i=1, 2)$ denote the knot obtained by blowing down the other unknotted component, i.e. performing a negative Rolfsen twist as in Figure \ref{rolfsen}.  
Then the $h$-function and $\nu^{+}$-invariant \cite[Definition 2.1]{HW} of $L'_{i}$ can be obtained from the $h$-function of $\L$.

\begin{proposition}
\label{prop:tau}
Let $\L=L_{1}\cup L_{2}$ be a nontrivial linking number zero L--space link with unknotted components, and let 
$L'_1$ and $L'_2$ be the knots obtained from $\L$ by applying a negative Rolfsen twist to $L_2$ and $L_1$ respectively.
Then $\nu^{+}(L'_i)=b_i+1$ for $i=1, 2$.
\end{proposition} 
Here, $b_1$ and $b_2$ are nonnegative numbers defined by $b_1=\max\{s_1: h(s_1,0)>0\}$ and $b_2=\max\{s_2: h(0,s_2)>0\}$. 
This allows us to determine, in terms of the $\nu^{+}$ invariants of $L'_1$ and $L'_2$, how large is `large enough' in order to guarantee that the surgery manifold is an L--space. 
\begin{theorem}
\label{thm:taubound}
Assume that $\L=L_1\cup L_2$ is a nontrivial L--space link with unknotted components and linking number zero. Then $S^3_{p_1,p_2}(\L)$ is an L--space if and only if $p_1>2\nu^{+}(L_1')-2$ and $p_2 > 2 \nu^{+}(L_2')-2$.
\end{theorem}

\begin{remark}
This gives a characterization of the unlink since it is the only 2-component L--space link with unknotted components, vanishing linking number and arbitrarily positive and negative L--space surgeries. For a general discussion about L--space surgeries on 2-component L--space links, we refer the reader to \cite{Liu19}. 
\end{remark}

The following corollary suggests that twisting along a homologically trivial unknotted component will almost always destroy the property of being an L--space link, in the sense that it puts strong constraints on the image knot $L_2'$. It is worth comparing Corollary \ref{cor: L2prime L space} with \cite[Corollary 1.6]{BakerMotegi}, which characterizes infinite twist families of tight fibered knots. Because L--space knots are necessarily tight fibered, Baker and Motegi's result shows that at most finitely many L--space knots can be produced by twisting along a homologically trivial unknot.
\begin{corollary}
\label{cor: L2prime L space}
Assume that $\L=L_1\cup L_2$ is a nontrivial L--space link with unknotted components and linking number zero. Then $L'_2$ is an L--space knot if and only if $(1,p_2)$ surgery on $\L$ is an L--space for sufficiently large $p_2$. By Theorem \ref{thm:taubound} this is equivalent to $b_1=0$ and $\nu^{+}(L'_1)=1$.
\end{corollary}
 
In section \ref{sec:relationships} we investigate the relationship of the $h$-function for two-component links with the Sato-Levine invariant $\beta(\L)$ and the Casson invariant $\lambda(S^3_\p(\L))$, and make explicit how to express these as linear combinations of the $h$-function of sublinks of $\L$. 

\begin{proposition} Let $\L=L_1\cup L_2$ be an L--space link of linking number zero.
Let $$h'(\mathbf{s})=h(\mathbf{s})-h_{1}(s_1)-h_{2}(s_2)$$ where $h, h_1$, and $h_2$ denote the $h$-functions of $\L, L_1$, and $L_2$.  Then:
\begin{enumerate}
	\item The Sato-Levine invariant of $\L$ equals
	$\beta(\L) = -\sum_{\mathbf{s}\in\Z^2}h'(\mathbf{s}). $ 
	\item Consider surgery coefficients $p_1, p_2 = \pm1$. The Casson invariant of $(p_1,p_2)$--surgery along $\L$ equals
	\[
		\lambda(S^3_{p_1, p_2}(\L)) = p_1p_2\sum_{\mathbf{s}\in\Z^2} h'(\mathbf{s}) +p_1\sum_{s_1\in\Z} h_1(s_1) + p_2\sum_{s_2\in\Z} h_2(s_2).
	\]
\end{enumerate}
\end{proposition}

Peters established a ``skein inequality" reminiscent of that for knot signature \cite[Theorem 1.4]{Peters} . We extend this to links as follows.

\begin{theorem}
Let $\L=L_1\cup\cdots\cup L_n$ be a link with all pairwise linking numbers zero. Given a diagram of $\L$ with a distinguished crossing $c$ on component $L_i$, let $D_+$ and $D_-$ denote the result of switching $c$ to positive and negative crossings, respectively. Then 
\[
	d(S^{3}_{1, \cdots, 1}(D_-)) - 2 \leq d(S^{3}_{1, \cdots, 1}(D_+)) \leq d(S^{3}_{1, \cdots, 1}(D_-)).
\]
\end{theorem}

We will also generalize Peters'  and Rasmussen's four-ball genus bounds to links with vanishing pairwise linking numbers.
Recall that the $n$ components of a link $\L=L_{1}\cup \cdots \cup L_{n}$ bound pairwise disjoint surfaces in $B^{4}$ if and only the pairwise linking numbers are all zero. In this case, we may define the smooth $4$-ball genus of $L$ as the minimum sum of genera $\sum_{i=1}^{n} g_{i}$, over all disjoint smooth embeddings of the surfaces $\Sigma_i$ bounding link components $L_i$, for $i=1, \cdots n$.

The following proposition is closely related to work of the second author in \cite{Liu}; this is explained in section \ref{sec:genusbounds}. 

\begin{proposition}
\label{prop:genusbound}
Let $\L\subset S^{3}$ denote an $n$-component link with pairwise vanishing linking numbers. Assume that $p_{i}>0$ for all $1\leq i \leq n$. Then 
\begin{equation}
\label{d-gen-ineq}
 d(S^{3}_{-p_1, \cdots, -p_n}(\L), \mft)\leq \sum_{i=1}^{n} d(L(-p_i, 1), t_i) +2f_{g_{i}}(t_i)
 \end{equation}
 and
 \begin{equation}
 \label{d-gen-ineq2}
-d(S^{3}_{p_1, \cdots, p_n}(\L), \mft)\leq \sum_{i=1}^{n} d(L(-p_i, 1), t_i) +2f_{g_{i}}(t_i).
 \end{equation}
Here the $\spinc$-structure $\mft$ is labelled by integers $(t_1, \cdots, t_n)$ where $-p_i/2\le t_i \le p_i/2$,  and $f_{g_{i}}: \mathbb{Z}\rightarrow \mathbb{Z}$ is defined as follows:
\begin{equation}
\label{def f}
f_{g_{i}}(t_{i}) = \left\{
        \begin{array}{ll}
           \left \lceil \dfrac{g_{i}-|t_{i}|}{2}\right\rceil & \quad |t_{i}|\leq g_{i} \\
            0 & \quad |t_{i}| > g_{i}
        \end{array}
    \right. 
\end{equation}

\end{proposition}
The $d$-invariant of $(\pm 1, \pm 1)$-surgery on the 2-bridge link $\L=b(8k, 4k+1)$ was computed by Y. Liu in \cite{LiuY1}. Together with this calculation, we are able to apply the genus bound \eqref{d-gen-ineq2} to recover the fact that such a link $\L$ has smooth four-genus one. We also demonstrate that this bound is sharp for Bing doubles of knots with positive $\tau$ invariant. For  more details, see section \ref{subsec:genusboundsexamples}.

Because Theorem \ref{thm:generalizedniwu} allows us to compute the $d$-invariants of $S^{3}_{\pm \p}(\L)$ for two-component L--space links, when we combine Theorem \ref{thm:generalizedniwu} with Proposition \ref{prop:genusbound} we have the following improved bound. 

\begin{theorem}
\label{thm:h-f}
Let $\L=L_{1}\cup L_{2}$ denote a two-component L--space link with vanishing linking number. Then for all $p_1, p_2>0$ and a $\spinc$-structure $\mft=(t_1, t_2)$ on $S^{3}_{p_1, p_2}$, we have 
\begin{equation}
\label{h-f-inequality}
 h(s_1, s_2) \leq f_{g_{1}}(t_{1})+f_{g_{2}}(t_{2})
\end{equation}
where  $-p_i/2\le t_i\le p_i/2$  and $(s_1, s_2)$ is a lattice point in the $\spinc$-structure $\mft$. 
\end{theorem}

\medskip
{\bf Organization of the paper.} Section \ref{sec:background} covers necessary background material. In subsection \ref{subsec:standard}, we introduce  standard 3-manifolds along with the definition and properties of the $d$-invariants for such manifolds. In subsection \ref{sec:hfunction}, we define the $h$-function of an oriented link $\L\subset S^{3}$ and review how to compute the $h$-function of an L--space link from its Alexander polynomial. Sections \ref{sec:linksurgery} and  \ref{subsec:dfromcells} are devoted to the generalized Ni-Wu $d$-invariant formula and its associated link surgery and cell complexes. In subsection \ref{subsec:knots surgery} we briefly review the surgery complex for knots, and in subsection \ref{subsec:truncation} we set up the Manolescu-Ozsv\'ath link surgery formula for links, and describe an associated cell complex and the truncation procedure. In section \ref{subsec:dfromcells} we prove Theorem \ref{thm:generalizedniwu} and the subsequent statements involving $\nu^{+}$. In section \ref{unknot components}, we classify L--space surgeries on L--space links with unknotted components and prove Theorem \ref{thm:taubound}. In section \ref{sec:relationships}, we represent the Sato-Levine invariant and Casson invariant of $S^{3}_{\pm1, \pm1}(\L)$ as linear combinations of the $h$-function for two-component L--space links with vanishing linking number. In section \ref{sec:skein}, we prove that the $d$-invariants of surgery 3-manifolds satisfy a skein inequality. In section \ref{sec:genusbounds}, we describe several bounds on the smooth four-genus of a link from the $d$-invariant and use this to establish the four-ball genera of several two-component links. 

\medskip
\textbf{Conventions.} In this article, we take singular homology coefficients in $\Z$ and Heegaard Floer homology coefficients in the field $\F=\Z/2\Z$ unless otherwise stated. We consider nonzero surgeries $S^{3}_{p_{1}, \cdots, p_{n}}(\L)$ on links $\L=L_{1}\cup \cdots \cup L_{n}$ in $S^{3}$, i.e. $p_{i}\neq 0$ for any $1\leq i \leq n$. 
Our convention on Dehn surgery is that $p$ surgery on the unknot produces the lens space $L(p, 1)$. We will primarily use the `minus' version of Heegaard Floer homology and adopt the convention that $d$-invariants are calculated from $HF^-(Y, \mft)$ and that $d^{-}(S^3)=0$. Section \ref{sec:background} contains further details on our degree conventions.


\section{Background}
\label{sec:background}

\subsection{$\spinc$-structures and $d$-invariants}
\label{subsec:spinc}

In this paper,  all the links are assumed to be oriented. We use $\L$ to denote a link in $S^{3}$, and  $L_{1}, \cdots, L_{n}$  to denote the link components. Then $\L_{1}$ and $\L_{2}$ denote different links in $S^{3}$, and $L_{1}$ and $L_{2}$ denote different components in the same link. Let $|\L|$ denote the number of components of $\L$. We denote  vectors in the $n$-dimensional lattice $\mathbb{Z}^{n}$ by bold letters. For two vectors $\bm{u}=(u_{1}, u_{2}, \cdots, u_{n})$ and $\bm{v}=(v_{1}, \cdots, v_{n})$ in $\mathbb{Z}^{n}$,  we write $\bm{u}\preceq \bm{v}$  if $u_{i}\leq v_{i}$ for each $1\leq i\leq n$, and $\bm{u}\prec \bm{v}$ if $\bm{u}\preceq \bm{v}$ and $\bm{u}\neq \bm{v}$. Let $\bm{e}_{i}$ be the vector in $\Z^{n}$ where the $i$-th entry is 1 and other entries are 0. For any subset $B\subset \lbrace 1, \cdots, n \rbrace$, let $\bm{e}_{B}=\sum_{i\in B} \bm{e}_{i}$.

Recall that in general, there is a non-canonical correspondence $\spinc(Y)\cong H^2(Y)$. For surgeries on links in $S^3$ we will require the following definition to parameterize $\spinc$-structures.
\begin{definition}
\label{Hfunctions}
For an oriented link $\L=L_{1}\cup \cdots \cup L_{n}\subset S^{3}$, define $\mathbb{H}(\L)$ to be the affine lattice over $\Z^{n}$:
$$\H(\L)=\oplus_{i=1}^{n}\H_{i}(\L), \quad \H_{i}(\L)=\Z+\dfrac{lk(L_{i}, \L\setminus L_{i})}{2}$$
where $lk(L_{i}, \L\setminus L_{i})$ denotes the linking number of $L_{i}$ and $\L\setminus L_{i}$. 
\end{definition} 

Suppose $\L$ has vanishing pairwise linking numbers. Then $\H(\L)=\Z^{n}$; we will assume this throughout the paper. Let $S^{3}_{p_1, \cdots, p_n}(\L)$ or $S^{3}_{\p}(\L)$ denote the surgery 3-manifold with integral surgery coefficients $\p = (p_1, \cdots, p_n)$. The quotient $\Z^{n}/ \Lambda \Z^{n}$ can be naturally identified with the space of Spin$^{c}$ structures on the surgery manifold $S^{3}_{p_{1}, \cdots, p_{n}}(\L)$,  where $\Lambda$ is the surgery matrix with diagonal entries $p_i$ and other entries 0. So  Spin$^{c}(\Sp)\cong \Z^{n}/ \Lambda \Z^{n}\cong \Z_{p_1}\oplus \cdots \oplus \Z_{p_n} \cong H^{2}(S^3_\p(\L))$. We therefore label Spin$^{c}$-structures $\mft$ on $\Sp$ as $(t_1, \cdots, t_n)$ such that  $-|p_i|/2 \leq t_i\leq |p_i|/2$.

For a rational homology sphere $Y$ with a Spin$^{c}$-structure $\mft$, the Heegaard Floer homology $HF^{+}(Y,\mft)$ is an absolutely graded $\F[U^{-1}]$-module, and its free part is isomorphic to $\F[U^{-1}]$. Likewise  $HF^{-}(Y,\mft)$ is an absolutely graded $\F[U]$-module. 
Given an oriented link $\L$ in $S^3$, one can also define the link Floer complex. 
An $n$-component link $\L$ induces $n$ filtrations on the Heegaard Floer complex $CF^{-}(S^3)$, and this filtration is indexed by the affine lattice $\H(\L)$. The link Floer homology $HFL^{-}(\L, \bm{s})$ is the homology of the associated graded complex with respect to this filtration, and is a module over $\F[U]$. 
We refer the reader to \cite{OS:Absolutely, ManOzs} for general background on Heegaard Floer and link Floer homology, and to \cite{BG} for a concise review relevant to our purposes. 

\begin{remark}
\label{rem:completion}
Following \cite{ManOzs}, we will sometimes need to work with the completed surgery complexes, which are defined as modules over $\F[[U]]$. For a rational homology sphere $Y$, the complex $CF^{-}(Y)$ and its completion over $\F[[U]]$
carry the same information.
\end{remark}

The $d$-invariant $d(Y, \mft$) is defined to be  the maximal degree of a non-torsion class $x\in HF^{-}(Y, \mft)$. In this article we adopt the convention that $d(S^3)=0$. This is consistent with the conventions of \cite{ManOzs, BG} but differs (by a shift of two) from that of \cite{OS:Absolutely}. 

We require the following statements on the $d$-invariant. 
\begin{proposition}\cite[Section 9]{OS:Absolutely}
\label{prop:OSdfacts}
Let $(W, \mfs): (Y_1, \mft_1) \rightarrow (Y_2, \mft_2 )$ be a $\spinc$ cobordism. 
\begin{enumerate}
	\item \label{fact1} If $W$ is negative definite, then $d(Y_2, \mft_2)- d(Y_1, \mft_1)\geq (c_1(\mfs)^2+b_2(W)) / 4$. 
	\item \label{fact2} If $W$ is a rational homology cobordism, then $d(Y_1, \mft_1) = d(Y_2, \mft_2)$.  
\end{enumerate} 
\end{proposition}

\begin{remark}
\label{monopole}
An equivalence of monopole Floer and Heegaard Floer homology has been established by work of Kutluhan-Lee-Taubes \cite{KLT1,KLT2,KLT3,KLT4,KLT5} and Colin-Ghiggini-Honda \cite{CGH1, CGH2, CGH3} and Taubes \cite{Taubes}. A further equivalence between monopole Floer and the $S^1$-equivariant homology of the Seiberg-Witten
Floer spectrum is proved in \cite{LidmanManolescu}. Following further work in \cite{Gripp,GrippHuang,Gardiner}, the absolute $\mathbb{Q}$-gradings of these theories agree. For rational homology spheres, 
\[
	d(Y, \mft) = -2h(Y, \mft) = 2\delta(Y, \mft),
\]
where $h(Y, \mft)$ is the Fr\o yshov invariant in monopole Floer homology \cite{KM, Froyshov} and $\delta(Y, \mft)$ is the analogous invariant of the Floer spectrum Seiberg-Witten theory \cite{Manolescu}. 
\end{remark}


\subsection{Standard 3-manifolds}
\label{subsec:standard}

In this subsection, we will introduce  $d$-invariants for standard 3-manifolds, and in particular, for circle bundles over oriented closed genus $g$ surfaces. 

Let $H$ be a finitely generated, free abelian group and $\Lambda^{\ast}(H)$ denote the exterior algebra of $H$. 
As in \cite[Section 9]{OS:Absolutely}, we say that $HF^\infty(Y)$ is standard if for each torsion $\spinc$ structure $\mft$,
\[
	HF^\infty(Y, \mft) \cong \Lambda^{\ast}H^1(Y;\Z) \otimes_\Z \F[U, U^{-1}]
\]
as $\Lambda^{\ast}H_{1}(Y; \Z)/ \textup{Tors} \otimes \F[U]$-modules. The group $\Lambda^* H^1(Y;\Z)$ is graded by requiring $gr (\Lambda^{b_1(Y)}H^1(Y;\Z)) = b_1(Y)/2$ and the fact that the action of $H_1(Y; \Z)/Tors$ by contraction drops gradings by $1$. For example, $\#^{n} S^2\times S^1$ has standard $HF^{\infty}$  \cite{OS:three}. 

For any $\Lambda^{\ast}(H)$-module $M$, we denote the kernel of the action of $\Lambda^{\ast}(H)$ on $M$ as 
$$\mathcal{K}M:=\{ x\in M \mid v\cdot x=0 \quad  \forall  \quad v\in H \}.$$
Let $\mathcal{I}$ denote the (two-sided) ideal in $\Lambda^{\ast}(H)$ generated by $H$. Define 
$$\mathcal{Q}M:= M/ (\mathcal{I} \cdot M). $$

For a standard 3-manifold $Y$,  we have the following induced maps:
\begin{equation}
\mathcal{K}(\pi):   \mathcal{K}HF^{\infty}(Y, \mft)\rightarrow \mathcal{K}HF^{+}(Y, \mft)
\end{equation}
and 
\begin{equation}
\mathcal{Q}(\pi): \mathcal{Q}HF^{\infty}(Y, \mft) \rightarrow \mathcal{Q}HF^{+}(Y, \mft).
\end{equation}

Define the \emph{bottom} and \emph{top correction terms} of $(Y, \mft)$ to be the minimal grading of any nonzero element in the image of $\mathcal{K}(\pi)$ and $\mathcal{Q}(\pi)$, denoted by $d_{bot}$ and $d_{top}$,  respectively \cite{Levine}. 
Levine and Ruberman established the following properties of $d_{top}$ and $d_{bot}$. 

\begin{proposition}\cite[Proposition 4.2]{Levine}
Let $Y$ be a closed oriented standard 3-manifold, and let $\mft$ be a torsion $\spinc$ structure on $Y$. Then 
$$d_{top}(Y, \mft)=-d_{bot}(-Y, \mft).$$
\end{proposition}

\begin{proposition}\cite[Proposition 4.3]{Levine}
Let $Y, Z$ be closed oriented standard 3-manifolds, and let $\mft, \mft'$ be torsion $\spinc$ structures on $Y, Z$ respectively. Then 
$$d_{bot}(Y\# Z, \mft\# \mft')=d_{bot}(Y, \mft)+d_{bot}(Z, \mft')$$
and 
$$d_{top}(Y\# Z, \mft\# \mft')=d_{top}(Y, \mft)+d_{top}(Z, \mft').$$
\end{proposition}

Let $B_n$ denote a circle bundle over a closed oriented genus $g$ surface with Euler characteristic $n\neq 0$. It can be obtained from $n$-framed surgery in $\#^{2g} S^{2}\times S^{1}$ along the ``Borromean knot." The torsion $\spinc$ structures on $B_{n}$ can be labelled by $-|n|/2\leq i \leq |n|/2$ \cite{Park, Ras}, though the labelling is not a bijection. A surgery exact triangle argument for the Borromean knot shows that
\[
	HF^{\infty}(B_{n}, i)\cong HF^{\infty}(\#^{2g} S^{2}\times S^{1}, \mft),
\]
where $\mft$ is the unique torsion $\spinc$ structure on $\#^{2g} (S^2\times S^1)$. Hence, $B_{n}$ is also standard for $n\neq 0$ \cite{Park, Ras}. 

The $d$-invariants for circle bundles $B_{n}$ have been computed in \cite{Park}.

\begin{theorem}\cite[Theorem 4.2.3]{Park}
\label{d-invaraint of circle bundle}
Let  $B_{-p}$ denote a circle bundle over a  closed oriented genus $g$ surface $\Sigma_{g}$ with Euler number $-p$. If $p>0$, then for any choice of $-p/2\leq i \leq p/2$ 
\[ d_{bot}(B_{p}, i)=-d_{top}(B_{-p}, i)=\phi(p, i)-g.\]
and
\[
d_{bot}(B_{-p}, i) = \left\{
        \begin{array}{ll}
            -\phi(p, i)-g & \quad \text{ if } |i|>g \\
            -\phi(p, i)-|i| & \quad  \text{ if } |i|\leq g \text{ and } g+i \text{ is even}\\
            -\phi(p, i)-|i|+1 & \quad \text{ if } |i|<g \text{ and } g+i \text{ is odd}, 
        \end{array}
    \right. 
\]

where
\[
	\phi(p, i)=d(L(p, 1), i)=-\max_{\{s\in \Z \mid s\equiv i (\text{mod } p)\}}  \dfrac{1}{4}  \left( 1-\dfrac{(p+2s)^{2}}{p} \right).
\]
\end{theorem}
 
\begin{remark}
For the rest of the paper, we use $\phi(p, i)$ to denote the $d$-invariant of the Spin$^{c}$ lens space $(L(p, 1), i)$ where $-p/2\leq i \leq p/2$ and $p>0$. For $p<0$, $\phi(p, i)=-\phi(-p, i)$. In this paper, we use the convention that $p$-surgery on the unknot yields the lens space $L(p, 1)$. 
\end{remark}

\begin{remark}
Observe that we can rewrite the formula in Theorem \ref{d-invaraint of circle bundle} using the function $f$ defined by \eqref{def f}:
\begin{equation}
\label{d circle bundle via f}
d_{bot}(B_{-p}, i)=-\phi(p, i)+2f_{g}(i)-g.
\end{equation}
\end{remark}

Ozsv\'ath and Szab\'o established the behaviour of the $d$-invariants of standard 3-manifolds under negative semi-definite $\spinc$-cobordisms. 

\begin{proposition}\cite[Theorem 9.15]{OS:Absolutely}
\label{prop:standard-inequality}
Let $Y$ be a three-manifold with standard $HF^\infty$, equipped with a torsion $\spinc$ structure $\mft$. Then for each negative semi-definite four-manifold $W$ which bounds $Y$ so that the restriction map $H^1(W)\rightarrow H^1(Y)$ is trivial, we have the inequality:
\begin{equation}
\label{eqn:standard-inequality}
	c_1(\mathfrak{s})^2 +  b_2^-(W) \leq 4d_{bot}(Y, \mft) +2b_1(Y)
\end{equation}
for all $\spinc$ structures $\mfs$ over $W$ whose restriction to $Y$ is $\mft$.
\end{proposition}


\subsection{The $h$-function and L--space links}
\label{sec:hfunction}

We review the $h$-function for oriented links $\L\subseteq S^{3}$, as defined by the first author and N\'emethi \cite{GN}.

A link $\L=L_{1}\cup \cdots \cup L_{n}$ in $S^{3}$ defines a filtration on the Floer complex $CF^{-}(S^{3})$, and the filtration is indexed by the $n$-dimensional lattice $\H(\L)$ (see Definition \ref{Hfunctions}). 
Given $\bm{s}=(s_{1}, \cdots, s_{n})\in \H(\L)$,  the \emph{generalized Heegaard Floer complex} $\A^{-}(\L, \bm{s}) \subset CF^{-}(S^3)$ is the $\F[[U]]$-module defined to be a subcomplex of $CF^{-}(S^{3})$ corresponding to the filtration indexed by $\bm{s}$ \cite{ManOzs} (here we implicitly completed $CF^{-}(S^{3})$ over $\F[[U]]$, see Remark \ref{rem:completion}). 

By the large surgery theorem \cite[Theorem 12.1]{ManOzs}, the homology of $\A^{-}(\L, \bm{s})$ is isomorphic to the Heegaard Floer homology of a large surgery on the link $\L$ equipped with some Spin$^{c}$-structure as an $\F[[U]]$-module. 
Thus the homology of $\A^{-}(\L, \bm{s})$ is non-canonically isomorphic to  a direct sum of one copy of $\F[[U]]$ and some $U$-torsion submodule, and so the following definition is well-defined.

\begin{definition}\cite[Definition 3.9]{BG}
\label{Hfunction}
For an oriented link $\L\subseteq S^{3}$, we define the $H$-function $H_{\L}(\bm{s})$ by saying that $-2H_{\L}(\bm{s})$ is the maximal homological degree of the free part of $H_{\ast}(\A^{-}(\L, \bm{s}))$ where $\bm{s}\in \H$. 
\end{definition}

\begin{remark}
We sometimes write $H_{\L}(\bm{s})$ as $H(\bm{s})$ for simplicity if there is no confusion in the context. 
\end{remark}
More specifically, the large surgery theorem of Manolescu-Ozsv\'ath \cite[Theorem 12.1]{ManOzs} implies that $-2H_{\L}(\bm{s})$ is the $d$-invariant of large surgery on $\L$, after some degree shift that depends on the surgery coefficient and $\bm{s}$ (see \cite[Section 10]{ManOzs}, \cite[Theorem 4.10]{BG}).
Note that  the $H$-function is a topological invariant of links in the three-sphere since it is defined in terms of the link invariant $CFL^{\infty}$.

Many practitioners of Heegaard Floer homology are more accustomed to working with the integer-valued knot invariants $V^+_s$ and $H^+_s$ of Ni and Wu \cite{NiWu}. For knots, 
$H_K(s)=V^+_s$. For example, the $H$-function of the left-handed trefoil is $H(s) = 0$ for $s\geq0$, $H(s) = -s$ for $s<0$.

We now list several properties of the $H$-function.
\begin{lemma}\cite[Proposition 3.10]{BG}
\label{h-function increase}
{\emph(Controlled growth)} For an oriented link $\L\subseteq S^{3}$, the $H$-function $H_{\L}(\bm{s})$ takes nonnegative values, and $H_{\L}(\bm{s}-\bm{e}_{i})=H_{\L}(\bm{s})$ or $H_{\L}(\bm{s}-\bm{e}_{i})=H_{\L}(\bm{s})+1$ where $\bm{s}\in \H$. 

\end{lemma}

\begin{lemma}\cite[Lemma 5.5]{LiuY2}
\label{h-function symmetry}
{\emph(Symmetry)} 
For an oriented $n$-component  link $\L\subseteq S^{3}$, the $H$-function satisfies $H(-\bm{s})=H(\bm{s})+\sum_{i=1}^{n} s_i$ where $\bm{s}=(s_1, \cdots, s_n)$.
\end{lemma}

Note that in \cite{LiuY2} the symmetry property is stated for L--space links, but the statement holds more generally. This is because the $H$-function is determined by the $d$-invariant of large surgery along the link and because $d$-invariants are preserved under $\spinc$-conjugation. See for example \cite[Lemma 2.5]{HLZ}.

\begin{lemma}\cite[Proposition 3.12]{BG}
\label{h-function bdy}
{\emph(Stabilization)}
For an oriented link $\L=L_1\cup \cdots \cup L_n \subseteq S^{3}$ with vanishing pairwise linking number, 
\[
	H_{\L}(s_1, \cdots, s_{i-1}, N, s_{i+1},  \cdots, s_n)=H_{\L\setminus L_i}(s_1, \cdots, s_{i-1}, s_{i+1}, \cdots, s_n)
\]
where $N$ is sufficiently large. 

\end{lemma}

For an $n$-component link $\L$ with vanishing pairwise linking numbers, $\H(\L)=\Z^{n}$. The $h$-function $h_{\L}(\bm{s})$  is defined as 
\[
	h_{\L}(\bm{s})=H_{\L}(\bm{s})-H_{O}(\bm{s}),
\]
where $h_\emptyset=0$, $O$ denotes the unlink with $n$ components, and $\bm{s}\in \Z^{n}$. Recall that for split links $\L$, the $H$-function 
$H(\L, \bm{s})=H_{L_{1}}(s_{1})+\cdots +H_{L_{n}}(s_{n})$ where $H_{L_{i}}(s_{i})$ is the $H$-function of the link component $L_{i}$, \cite[Proposition 3.11]{BG}. 
Then $H_{O}(\bm{s})=H(s_{1})+\cdots H(s_{n})$ where $H(s_{i})$ denotes the $H$-function of the unknot. More precisely, $H_{O}(\bm{s})=\sum_{i=1}^{n}(|s_{i}|-s_{i})/2$ by \cite[Section 2.6]{OS:Integer}.  
Hence $H_{\L}(\bm{s})=h_{\L}(\bm{s})$ for all $\bm{s}\succeq \bm{0}$. 
By Lemma \ref{h-function symmetry} we get
\begin{equation}
\label{eq: h symmetry}
h(-\bm{s})=h(\bm{s}).
\end{equation}

\begin{lemma}
\label{lem: h increases}
The function $h$ is non-decreasing towards the origin. That is, $h(\bm{s}-\bm{e}_i)\ge h(\bm{s})$ if $s_i>0$ and $h(\bm{s}-\bm{e}_i)\le h(\bm{s})$ if $s_i\le 0$.  
\end{lemma}

\begin{proof}
 If $s_i>0$ then $H_O(s_i)=H_O(s_i-1)=0$, so
$$
h(\bm{s})-h(\bm{s}-\bm{e}_i)=H(\bm{s})-H(\bm{s}-\bm{e}_i)\le 0.
$$
If $s_i\le 0$ then $H_O(s_i)=-s_i$ and $H_O(s_i-1)=1-s_i$, so
$$
h(\bm{s})-h(\bm{s}-\bm{e}_i)=H(\bm{s})-H(\bm{s}-\bm{e}_i)+1\ge 0.
$$
\end{proof}

\begin{corollary}
\label{h nonnegative}
For all $\bm{s}$ one has $h(\bm{s})\ge 0$.
\end{corollary}

\begin{proof}
We prove it by induction on the number $n$ of components of $\L$. If $n=0$, it is clear.
Assume that we proved the statement for $n-1$. Observe that by Lemma \ref{h-function bdy} for $s_i\gg 0$ we have 
$h(\bm{s})=h_{\L\setminus L_i}(\bm{s})\ge 0$. For $s_i\ll 0$ by \eqref{eq: h symmetry} we have
$$
h(\bm{s})=h(-\bm{s})=h_{\L\setminus L_i}(-\bm{s})\ge 0.
$$
Now by Lemma \ref{lem: h increases} we have $h(\bm{s})\ge 0$ for all $\bm{s}$.
\end{proof}

In \cite{OS:lenspaces}, Ozsv\'ath and Szab\'o introduced the concept of L--spaces. 

\begin{definition}
A 3-manifold $Y$ is an L--space if it is a rational homology sphere and its Heegaard Floer homology has minimal possible rank: for any Spin$^{c}$-structure $\mathfrak{s}$, $\widehat{HF}(Y, \mathfrak{s})=\F$ or, equivalently,   $HF^{-}(Y, \mathfrak{s})$ is a free $\F[U]$-module of rank 1. 
\end{definition}

\begin{definition}\cite{GN,LiuY2}
\label{definition of L--space link}
An oriented $n$-component link $\L\subset S^{3}$ is an L--space link if there exists  $\bm{0}\prec \bm{p}\in \mathbb{Z}^{n}$ such that the surgery manifold $S^{3}_{\bm{q}}(\L)$ is an L--space for any $\bm{q}\succeq \bm{p}$. 
\end{definition}

We list some useful properties of L--space links:

\begin{theorem}\cite{LiuY2}
\label{l-space link cond}
(a) Every sublink of an L--space link is an L--space link.

(b) A link is an L--space link if and only if  for all $\bm{s}$ one has $H_{\ast}(\A^{-}(\L, \bm{s}))=\F[[U]]$.

(c) Assume that for some $\bm{p}$ the surgery $S^{3}_{\bm{p}}(L)$ is an L--space. In addition, assume that for all sublinks $\L'\subset \L$ the surgeries $S^{3}_{\bm{p}}(\L')$ are L--spaces too, and the framing matrix $\Lambda$ is positive definite.
Then for all $\bm{q}\succeq \bm{p}$ the surgery manifold $S^{3}_{\bm{q}}(\L)$ is an L--space, and so $\L$ is an L--space link.
\end{theorem}

\begin{remark}
If all pairwise linking numbers between the components of $\L$ vanish, then $\Lambda$ is positive definite if and only if all $p_i>0$.
Therefore for (c) one needs to assume that there exist positive $p_i$ such that $S^{3}_{\bm{p}}(\L')$ is an L--space for any sublink $\L'$.
\end{remark}

For L--space links, the $H$-function can be computed from the multi-variable Alexander polynomial.
Indeed, by (b) and the inclusion-exclusion formula, one can write
\begin{equation}
\label{computation of h-function 1}
\chi(HFL^{-}(\L, \bm{s}))=\sum_{B\subset \lbrace 1, \cdots, n \rbrace}(-1)^{|B|-1}H_{\L}(\bm{s}-\bm{e}_{B}),
\end{equation}
as in \cite[(3.14)]{BG}. The Euler characteristic $\chi(HFL^{-}(\L, \bm{s}))$ was computed in \cite{OS:linkpoly},
\begin{equation}
\label{computation 3}
\tilde{\Delta}(t_{1}, \cdots, t_{n})=\sum_{\bm{s}\in \H(\L)}\chi(HFL^{-}(\L, \bm{s}))t_{1}^{s_{1}}\cdots t_{n}^{s_{n}}
\end{equation}
where $\bm{s}=(s_{1}, \cdots, s_{n})$, and
\begin{equation}
\label{mva}
\widetilde{\Delta}_{\L}(t_{1}, \cdots, t_{n}): = \left\{
        \begin{array}{ll}
           (t_{1}\cdots t_{n})^{1/2} \Delta_{\L}(t_{1}, \cdots, t_{n}) & \quad \textup{if } n >1, \\
            \Delta_{\L}(t)/(1-t^{-1}) & \quad  \textup{if } n=1. 
        \end{array}
    \right. 
\end{equation}

\begin{remark}
Here we expand the rational function as power series in $t^{-1}$, assuming that the exponents are bounded in positive direction. The Alexander polynomials are normalized so that they are symmetric about the origin. This still leaves out the sign ambiguity which can be resolved for L--space links by requiring that $H(s)\ge 0$ for all $s$.
\end{remark}

One can regard \eqref{computation of h-function 1} as a system of linear equations for $H(s)$ and solve it explicitly 
using the values of the $H$-function for sublinks as the boundary conditions. We refer to \cite{BG,GN} for general formulas,
and consider only links with one and two components here.

For $n=1$ the equation \eqref{computation of h-function 1} has the form 
\[
\chi(HFL^{-}(\L, s))=H(s-1)-H(s),
\]
so
\[
H(s)=\sum_{s'>s}\chi(HFL^{-}(\L, s')),\ \sum_{s}t^sH(s)=t^{-1}\Delta_{\L}(t)/(1-t^{-1})^2.
\]

For $n=2$ the equation \eqref{computation of h-function 1} has the form 
\begin{equation}
\label{chi from H 2 comp}
\chi(HFL^{-}(\L, \bm{s}))=-H(s_1-1,s_2-1)+H(s_1-1,s_2)+H(s_1,s_2-1)-H(s_1,s_2).\end{equation}

\begin{lemma}
Suppose that $L_1$ and $L_2$ are unknots and $lk(L_1,L_2)=0$, then 
\begin{equation}
\label{h from Alexander}
\sum_{s_1,s_2}t_1^{s_1}t_2^{s_2}h(s_1,s_2)=-\frac{t_1^{-1}t_2^{-1}}{(1-t_1^{-1})(1-t_2^{-1})}\widetilde{\Delta}(t_1,t_2).
\end{equation}
\end{lemma}

\begin{proof}
By Lemma \ref{h-function bdy} for sufficiently large $N$ we have $H(s_1,N)=H_1(s_1)$ and $H(N, s_2)=H_2(s_2)$ .
By \eqref{chi from H 2 comp} we get
\[
H(s_1,s_2)-H_1(s_1)-H_2(s_2)=H(s_1,s_2)-H(s_1,N)-H(N,s_2)=
\]
\[
-\sum_{\bm{s}'\succeq \bm{s}+\bm{1}}\chi(HFL^{-}(\L, \bm{s}')).
\]
Since $L_1$ and $L_2$ are unknots, we get $h(s_1,s_2)=H(s_1,s_2)-H_1(s_1)-H_2(s_2)$ and 
\[
\sum_{s_1,s_2}t_1^{s_1}t_2^{s_2}h(s_1,s_2)=-\sum_{s_{1}, s_{2}}\sum_{\bm{s}'\succeq \bm{s}+\bm{1}}t_1^{s_1}t_2^{s_2}\chi(HFL^{-}(\L, \bm{s}'))=
\]
\[
-\frac{t_1^{-1}t_2^{-1}}{(1-t_1^{-1})(1-t_2^{-1})}\sum_{\bm{s}'}t_1^{s'_1}t_2^{s'_2}\chi(HFL^{-}(\L, \bm{s}'))=
-\frac{t_1^{-1}t_2^{-1}}{(1-t_1^{-1})(1-t_2^{-1})}\widetilde{\Delta}(t_1,t_2).
\]
\end{proof}

\begin{example}
\label{wh H}
The (symmetric) Alexander polynomial of the Whitehead link equals 
\[
\Delta(t_1,t_2)=-(t_1^{1/2}-t_1^{-1/2})(t_2^{1/2}-t_2^{-1/2}),
\]
so
\[
\widetilde{\Delta}(t_1,t_2)=(t_1t_2)^{1/2}\Delta(t_1,t_2)=-(t_1-1)(t_2-1).
\]
The $H$-function has the following values:

\begin{center}
\begin{tikzpicture}
\draw (1,0)--(1,5);
\draw (2,0)--(2,5);
\draw (3,0)--(3,5);
\draw (4,0)--(4,5);
\draw (0,1)--(5,1);
\draw (0,2)--(5,2);
\draw (0,3)--(5,3);
\draw (0,4)--(5,4);
\draw (0.5,4.5) node {2};
\draw (1.5,4.5) node {1};
\draw (2.5,4.5) node {0};
\draw (3.5,4.5) node {0};
\draw (4.5,4.5) node {0};
\draw (0.5,3.5) node {2};
\draw (1.5,3.5) node {1};
\draw (2.5,3.5) node {0};
\draw (3.5,3.5) node {0};
\draw (4.5,3.5) node {0};
\draw (0.5,2.5) node {2};
\draw (1.5,2.5) node {1};
\draw (2.5,2.5) node {1};
\draw (3.5,2.5) node {0};
\draw (4.5,2.5) node {0};
\draw (0.5,1.5) node {3};
\draw (1.5,1.5) node {2};
\draw (2.5,1.5) node {1};
\draw (3.5,1.5) node {1};
\draw (4.5,1.5) node {1};
\draw (0.5,0.5) node {4};
\draw (1.5,0.5) node {3};
\draw (2.5,0.5) node {2};
\draw (3.5,0.5) node {2};
\draw (4.5,0.5) node {2};
\draw [->,dotted] (0,2.5)--(5,2.5);
\draw [->,dotted] (2.5,0)--(2.5,5);
\draw (5,2.7) node {$s_1$};
\draw (2.3,5) node {$s_2$};
\end{tikzpicture}
\end{center}
One can  check that \eqref{chi from H 2 comp} is satisfied for all $(s_1,s_2)$. Also, 
\[
h(s_1,s_2)=\begin{cases}
1\ \qquad\text{if}\ (s_1,s_2)=(0,0)\\
0\ \qquad\text{otherwise},\\
\end{cases}
\]
which agrees with \eqref{h from Alexander}.
\end{example}

\begin{lemma}
If for an L--space link $\L$ one has $h(0,0)=0$ then $\L$ is the unlink. 
\end{lemma}

\begin{proof}
If $h(0,0)=0$ then by Lemma \ref{lem: h increases} we have $h(s_1,s_2)=0$ for all $s_1,s_2$. The rest of the proof follows from \cite[Theorem 1.3]{Liu}.
\end{proof}

For example, the H-function, and consequently $\widehat{HFL}$ and the Thurston norm of the link complement of an L--space link of two-components may be calculated from the Alexander polynomial, albeit with a nontrivial spectral sequence argument, as in \cite{LiuB2}.


\section{Surgery formula and truncations}
\label{sec:linksurgery}

\subsection{Surgery for knots}
\label{subsec:knots surgery}

In this subsection we review the ``mapping cone'' complex for knots \cite{OS:Integer}, and its finite rank truncation. We will present it in an algebraic and graphical form ready for generalization to links. 
Let $K$ be a knot in $S^3$ and let $p\in \Z$.

For each $s\in \Z$ we consider complexes $\A^0_s:=\A^{-}(K,s)$, and $\A^1_s=\A^{-}(\emptyset)$. The surgery complex is defined as
$$
\C=\prod_{s} \C_s,\ \C_s=\A^0_s+\A^1_s.
$$
The differential on $\C$ is induced by an internal differential $\Phi^\emptyset$ in $\A^0_s,\A^1_s$, and two types of chain maps, 
$\Phi^{+}_s:\A^0_s\to \A^1_s$, $\Phi^{-}_s:\A^0_s\to \A^1_{s+p}$.
Then $D_s=\left( \begin{array}{cc} \Phi^\emptyset & 0 \\ \Phi^{+}_s+\Phi^{-}_s & \Phi^\emptyset \end{array} \right)$. 
The complex $(\C, D)$ is usually represented with a zig-zag diagram in which we omit the internal differential $\Phi^\emptyset$,
\begin{equation}
\label{zigzag}
\xymatrix@C=15pt@R=12pt{
\cdots \ar@{.>}[dr]_{h} & \A^0_{-b}\ar@{.>}[d]_v \ar[dr]_h & \A^0_{-b+p}\ar[d]_v\ar[dr]_h  & \cdots\ar[dr]_h  &\A^0_{s} \ar[d]_v \ar[dr]_h & \A^0_{s+p} \ar[d]_v \ar[dr]_h & \cdots \ar[dr]_h & \A^0_{b} \ar[d]_v \ar@{.>}[dr]_{h} & \cdots\\
\cdots & \A^1_{-b} & \A^1_{-b+p} &\cdots & \A^1_{s} & \A^1_{s+p} & \ldots & \A^1_{b}&\cdots
}
\end{equation}
Here the vertical maps are given by $\Phi^{+}_s$ and the sloped maps by $\Phi^{-}_s$.
We instead present the complex $\C$ graphically as follows:
for each $s$ we represent $\C_s$ as a circle at a point $s$ containing two dots representing $\A^0_s$ and $\A^1_s$. The internal differential and $\Phi^{+}_s$ act within each circle, while $\Phi^{-}_s$ jumps between different circles. To avoid cluttering we do not draw the differentials in this picture. See Figure \ref{knot1}.

One can choose a sufficiently large positive integer $b$ such that for $s>b$ the map $\Phi^{+}_s$ is a quasi-isomorphism, and for $s<-b$ the map $\Phi^{-}_s$ is a quasi-isomorphism. The first condition means that we can erase all  circles (and all dots inside them) to the right of $b$ without changing the homotopy type of $\C$. The second condition is more subtle and depends on the sign of the surgery coefficient $p$. 

\begin{figure}[H]
\centering
\includegraphics[width=\textwidth]{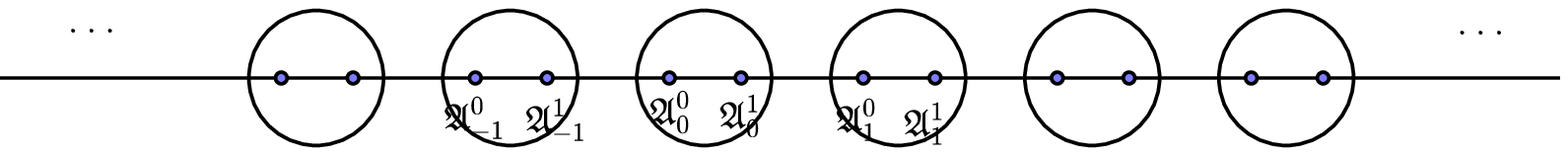}
\caption{The surgery complex $\C$ for a knot.}
\label{knot1}
\end{figure}

If $p>0$, we can use $\Phi^{-}_s$ to contract $\A^0_s$ with $\A^1_{s+p}$ for $s<-b$. By applying all these contractions at once, we erase all $\A^0_s$ for $s<-b$ and all $\A^1_{s+p}$ for $s<p-b$. As a result, graphically we will have a width $p$ interval $[-b,p-b)$ where each circle contains only $\A^0_s$, and a long interval $[p-b,b]$ where each circle contains both subcomplexes. See Figure \ref{knot2}.
\begin{figure}[H]
\centering
\includegraphics[width=\textwidth]{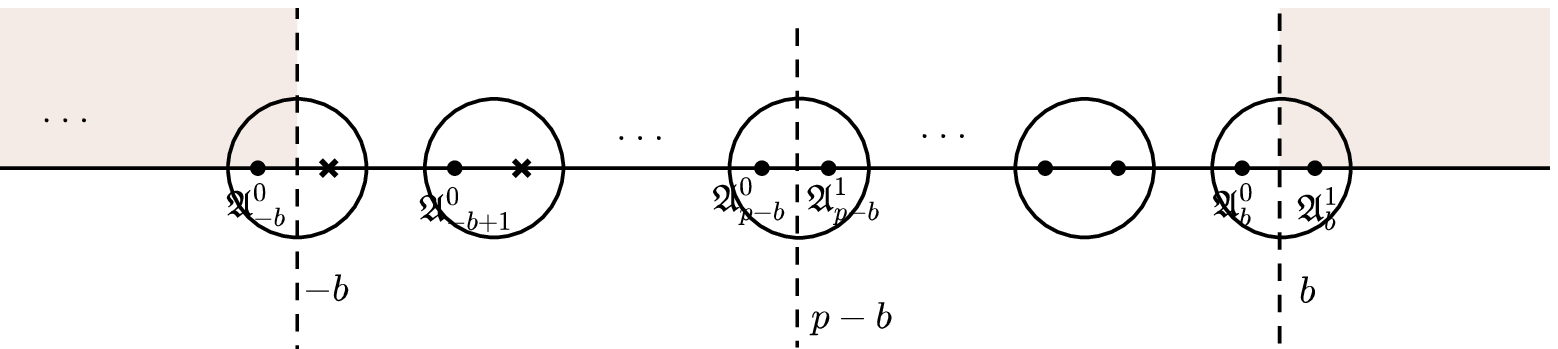}
\caption{The complex $\C$ after contraction when $p>0$. }
\label{knot2}
\end{figure}

If $p<0$, a similar argument shows that we will have a width $p$ interval $[p-b,-b)$ where each circle contains only $\A^1_s$, and a long interval $[-b,b]$ where each circle contains both subcomplexes. Note that in both cases in each $\spinc$ structure there is exactly one half-empty circle and a lot of full circles. Denote the truncated complex by $\C_b$. See Figure \ref{knot3}.
\begin{figure}[H]
\centering
\includegraphics[width=\textwidth]{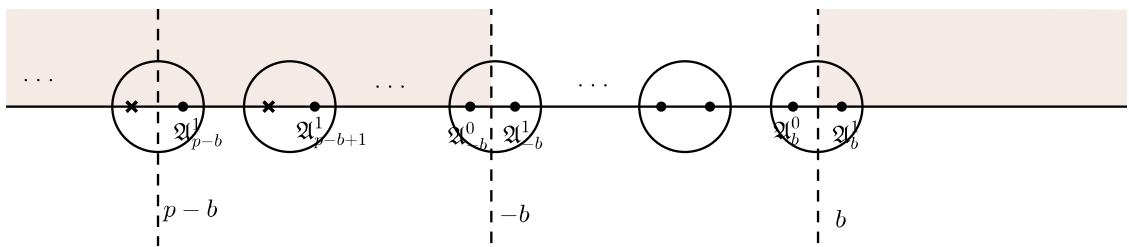}
\caption{The complex $\C$ after contraction when $p<0$. }
\label{knot3}
\end{figure}

Next, we would like to match $\A^0_s$ and $\A^1_s$ in $\C_b$ with the cells in a quotient or sub-complex $\CW(p,i,b)$ of a finite 1-dimensional CW complex. 
Each $\A^0_s$ corresponds to a 1-cell, and $\A^1_s$ to a 0-cell, and the boundary maps correspond to $\Phi^{\pm}_s$. 
The complexes corresponding to the previous two pictures are comprised of disjoint unions of $|p|$ intervals. 
Depending on the sign of $p$, each connected component is identified with one of the interval on the line subdivided by integer points pictured in Figure \ref{knot4}.

More specifically, for $p>0$ and each $\spinc$-structure $i$ (identified with a remainder modulo $|p|$), 
the complex $\CW(p,i,b)$ has one more 1-cell than 0-cell and can be identified with an open subdivided interval. We think of this as the closed subdivided interval $R$ with its two boundary cells $\partial R$ erased. The homology of $\CW(p,i,b)$ over $\F$ is $H_*(R, \partial R)\cong \F$, generated by the the sum of all 1-cells. 

For $p<0$ we have instead one more 0-cell than 1-cell. The complex $\CW(p,i,b)$ is now a closed interval $R$ with no boundary cells erased. The homology of $\CW(p,i,b)$ is $H_*(R,\emptyset) \cong \F$, generated by the class of a 0-cell.

\begin{figure}[H]
\centering
\includegraphics[width=3.8in]{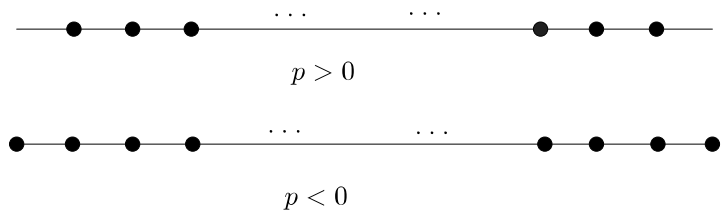}
\caption{The  complex $\CW(p, i, b)$.}
\label{knot4}
\end{figure}

So far, all of this is really just a rephrasing of the mapping cone formula of \cite{OS:Integer}. 
However, we will see that such pictures are easier to handle for more components, and the topology of the complexes $\CW(p,i,b)$ plays an important role. We will use this observation later in section \ref{subsec:dfromcells}.

\subsection{Truncation for 2-component L--space links }
\label{subsec:truncation}

We first review the Manolescu-Ozsv\'ath link surgery complex \cite{ManOzs} for oriented $2$-component  links $\L=L_{1}\cup L_{2}$ with vanishing linking number. Let $\mathcal{H}^{\L}=(\Sigma, \bm{\alpha}, \bm{\beta}, \bm{w}, \bm{z})$ be an admissible, generic, multi-pointed Heegaard diagram for $\L$.  Note that $\H(\L)\cong \Z^{2}$. 

For any sublink $M\subseteq \L$, set $N=\L-M$. We define a map 
\[
	\psi^{M}: \Z^{|\L|}\rightarrow \Z^{|N|}
\]
to be  the projection to the components corresponding to $L_{i}\subseteq N$. 
For sublinks $M\subseteq \L$, we use $\bH^{\L-M}$ to denote the Heegaard diagram of $\L-M$ obtained from $\bH^{\L}$ by forgetting the $z$ basepoints on the sublink $M$. The diagram $\bH^{\L-M}$ is associated with the generalized Floer complex $\A^{-}(\bH^{\L-M}, \psi^{M}(\bm{s})).$ 

In general, the surgery complex is complicated. For 2-component links with vanishing linking numbers, we describe the chain complex and its differential in detail. For the surgery matrix, we write  
\[
\Lambda=\begin{pmatrix}
p_{1} & 0 \\
0 & p_{2}
\end{pmatrix}.
\]
For a link $\L=L_{1}\cup L_{2}$, a two digit binary superscript is used to keep track of which link components are forgotten.
Let $\A^{00}_{\bm{s}}=\A^{-}(\bH^{\L}, \bm{s})$, $\A^{01}_{\bm{s}}=\A^{-}(\bH^{\L-L_{2}}, s_{1})$, $\A^{10}_{\bm{s}}=\A^{-}(\bH^{\L-L_{1}}, s_{2})$ and $\A^{11}_{\bm{s}}=\A^{-}(\bH^{\L-L_{1}-L_{2}}, \varnothing)$ where $\bm{s}=(s_{1}, s_{2})\in \Z^{2}$. Let 
$$\C_{\bm{s}}=\bigoplus_{\varepsilon_{1}, \varepsilon_{2}\in \lbrace 0, 1 \rbrace} \A^{\varepsilon_{1}\varepsilon_{2}}_{\bm{s}}.$$
The surgery complex is defined as
$$\C(\bH^{\L}, \Lambda)=\prod_{\bm{s}\in \Z^{2}} \C_{\bm{s}}.$$
The differential in the  complex is defined as follows. Consider sublinks $\varnothing, \pm L_{1}, \pm L_{2}$ and $\pm L_{1}\pm L_{2}$ where $\pm$ denotes whether or not the orientation of the sublink is the same as the one induced from  $\L$. Based on \cite{ManOzs}, 
we have the following maps, where $\Phi_{\bm{s}}^{\varnothing}$ is the internal differential on any chain complex $\A^{\varepsilon_{1}\varepsilon_{2}}_{\bm{s}}$. 
\begin{eqnarray}
\label{maps}
\begin{aligned}
\Phi^{L_{1}}_{\bm{s}}: \A^{00}_{\bm{s}}\rightarrow \A^{10}_{\bm{s}}, \quad &\Phi^{-L_{1}}_{\bm{s}}: \A^{00}_{\bm{s}}\rightarrow \A^{10}_{\bm{s}+\Lambda_{1}}, \\
\Phi^{L_{2}}_{\bm{s}}: \A^{00}_{\bm{s}}\rightarrow \A^{01}_{\bm{s}}, \quad &\Phi^{-L_{2}}_{\bm{s}}: \A^{00}_{\bm{s}}\rightarrow \A^{01}_{\bm{s}+\Lambda_{2}}, \\
\Phi^{L_{1}}_{s_{1}}: \A^{01}_{\bm{s}}\rightarrow \A^{11}_{\bm{s}}, \quad &\Phi^{-L_{1}}_{s_{1}}: \A^{01}_{\bm{s}}\rightarrow \A^{11}_{\bm{s}+\Lambda_{1}}, \\
\Phi^{L_{2}}_{s_{2}}: \A^{10}_{\bm{s}}\rightarrow \A^{11}_{\bm{s}}, \quad &\Phi^{-L_{2}}_{s_{2}}: \A^{10}_{\bm{s}}\rightarrow \A^{11}_{\bm{s}+\Lambda_{2}}, 
\end{aligned}
\end{eqnarray}
where $\Lambda_{i}$ is the $i$-th column of $\Lambda$.  
In addition, there are ``higher'' differentials 
\begin{eqnarray}
\label{higher maps}
\begin{aligned}
\Phi^{L_{1}+L_2}_{\bm{s}}: \A^{00}_{\bm{s}}\rightarrow \A^{11}_{\bm{s}}, \quad \Phi^{L_{1}-L_2}_{\bm{s}}: \A^{00}_{\bm{s}}\rightarrow \A^{11}_{\bm{s}+\Lambda_2}, \\
\Phi^{-L_{1}+L_2}_{\bm{s}}: \A^{00}_{\bm{s}}\rightarrow \A^{11}_{\bm{s}+\Lambda_1}, \quad \Phi^{-L_{1}-L_2}_{\bm{s}}: \A^{00}_{\bm{s}}\rightarrow \A^{11}_{\bm{s}+\Lambda_1+\Lambda_2}.
\end{aligned}
\end{eqnarray}
Let 
\[
	D_{\bm{s}}=\Phi^{\varnothing}_{\bm{s}}+\Phi^{\pm L_{1}}_{\bm{s}}+\Phi^{\pm L_{2}}_{\bm{s}}+\Phi^{\pm L_{1}}_{s_{1}}+\Phi^{\pm L_{2}}_{s_{2}}+\Phi^{\pm L_{1}\pm L_{2}}_{\bm{s}},
\]
and let $D=\prod_{\bm{s}\in \Z^{2}} D_{\bm{s}}$. Then $(\C(\bH^{\L}, \Lambda), D)$ is the Manolescu-Ozsv\'ath surgery complex. 

 The surgery complex naturally splits as a direct sum corresponding to the Spin$^c$-structures. The Spin$^c$-structures on $S^{3}_{\Lambda}(\L)$ are identified with $\H(\L)/ H(\L, \Lambda)\cong \Z_{p_1}\times \Z_{p_2}$ where $H(\L, \Lambda)$ is the subspace spanned  by $\Lambda$. For $\mft \in \H(\L)/ H(\L, \Lambda)$, choose $\bm{s}=(s_{1}, s_{2})$ corresponding with $\mft$ and let 
\[
	\C(\Lambda, \mft)=\bigoplus_{i, j\in \Z} \C_{\bm{s}+i\Lambda_{1}+j\Lambda_{2}}.
\]
Then by \cite{ManOzs}, 
\[
	HF^{-}(S^{3}_{\Lambda}(\L), \mft)\cong H_{\ast}(\C(\Lambda, \mft), D)
\]
up to some grading shift.

Now we review the truncation of the surgery complex $(\C(\bH^{\L}, \Lambda), D)$ \cite{ManOzs}, which mimics the truncation of the mapping cone for knots. 

\begin{lemma}\cite[Lemma 10.1]{ManOzs}
There exists a constant $b>0$ such that for any $i=1, 2$, and for any sublink $M\subset L$ not containing the component $L_{i}$, the chain map 
\[
	\Phi^{\pm L_{i}}_{\psi^{M}(\bm{s})}: \mathfrak{A}^{-}(\mathcal{H}^{L-M}, \psi^{M}(\bm{s}))\rightarrow \A^{-}(\bH^{L-M-L_{i}}, \psi^{M\cup L_{i}}(\bm{s}))
\]
induces an isomorphism on homology provided that either 
\begin{itemize}
\item $\bm{s}\in \Z^{2}$ is such that $s_{i}>b$, and $L_{i}$ is given the orientation induced from $L$; or 

\item $\bm{s}\in \Z^{2}$ is such that $s_{i}<-b$, and $L_{i}$ is given the orientation opposite to the one induced from $L$. 
\end{itemize}

\end{lemma}

\begin{figure}[H]
\centering
\includegraphics[width=\textwidth]{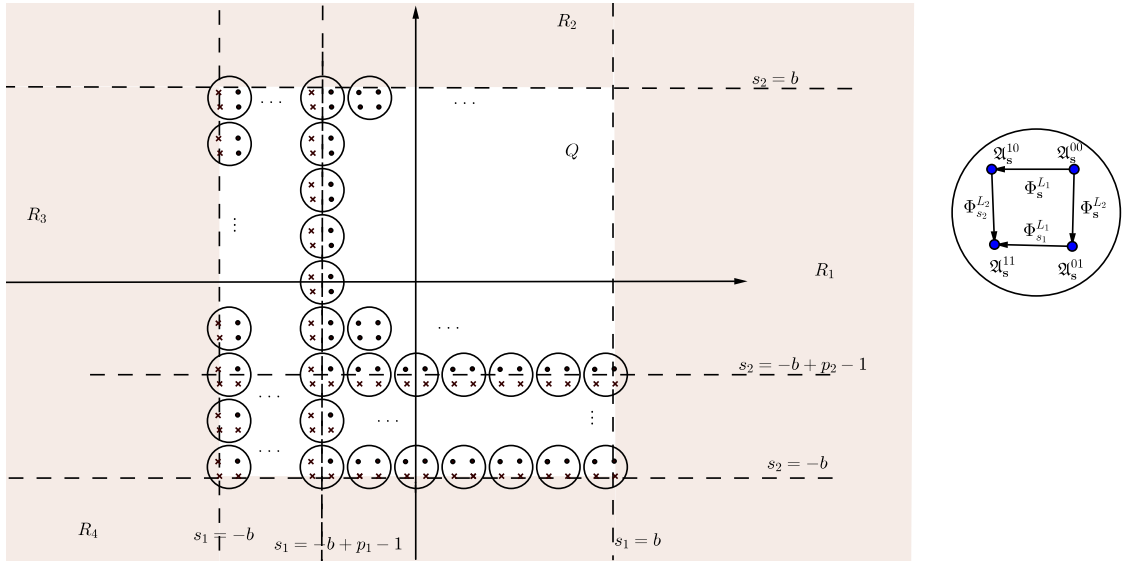}
\caption{ Truncated complex for $p_1, p_2>0$ \label{positive}}
\end{figure}

Without loss of generality, we will assume that
\[
b>\max(|p_1|,|p_2|).
\]
We consider five regions on the plane:
\[
Q=\{|s_1|\le b,|s_2|\le b\},\ R_1=\{s_1>b, s_2\leq b\},\ R_2=\{s_1\geq -b, s_2>b\},
\]
\[
 R_3=\{s_1<-b, s_2\geq -b\},\ R_4=\{s_1\leq b, s_2<-b\}.\  
\]

\begin{remark}
\label{rem: different truncation}
One can also use different constants $b_1,b_2$ to truncate the complex in vertical and in horizontal directions.
As a result, the rectangle $Q$ would be bounded by the lines $s_1=\pm b_1, s_2=\pm b_2$. All results below 
hold unchanged in this more general case.
\end{remark}

Depending on the signs of $p_{1}$ and $p_{2}$, the surgery complex may truncated as follows (see also the detailed case analysis of \cite[Section 10]{ManOzs}).

\textbf{Case 1}: $p_{1}>0, p_{2}>0$. 
In this case, let $\mathcal{C}_{R_{1}\cup R_{2}}$ be the subcomplex of $\C(\bH^{\L}, \Lambda)$ consisting of those terms $\A^{\varepsilon_{1}\varepsilon_{2}}_{\bm{s}}$ supported in $R_{1}\cup R_{2}$. The subcomplex $\mathcal{C}_{R_{1}\cup R_{2}}$ is acyclic \cite{ManOzs}. In the quotient complex $\C/\C_{R_{1}\cup R_{2}}$, define a subcomplex $\C_{R_{3}\cup R_{4}}$ consisting of those terms $\A_{\bm{s}}^{\varepsilon_{1}\varepsilon_{2}}$ with the property that $\bm{s}-\varepsilon_{1}\Lambda_{1}-\varepsilon_{2}\Lambda_{2}\in R_{3}\cup R_{4}$. Let $\C_{Q}$ be the quotient of $\C/\C_{R_{1}\cup R_{2}}$ by $\C_{R_3\cup R_4}$. Then $\C_{Q}$ is quasi-isomorphic to the original complex $\C(\bH^{\L}, \Lambda)$, and  $\C_{Q}$ consists of dots inside the box  indicated as in Figure \ref{positive}.

\begin{figure}[H]
\centering
\includegraphics[width=\textwidth]{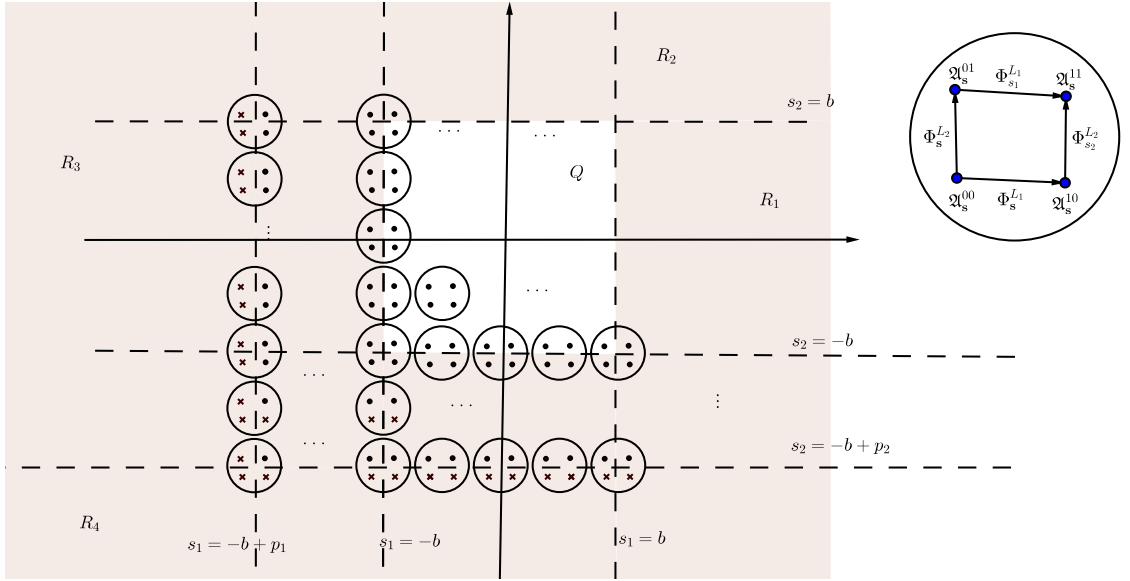}
\caption{Truncated complex for $p_1, p_2<0$ \label{negative}}
\end{figure}

\textbf{Case 2}: $p_{1}<0, p_{2}<0$.
This is similar to Case 1, except that $\C_{R_{1}\cup R_{2}}$ and $\C_{R_{3}\cup R_{4}}$ are now quotient complexes, and $\C_{Q}$ is a subcomplex as shown in Figure \ref{negative}. Note that $\C_Q$ contains all the solid dots pictured, including those outside of box $Q$.

\textbf{Case 3}: $p_{1}>0, p_{2}<0$.
First define two acyclic subcomplexes: one is $\C_{R_{1}}$, which consists of terms $\A_{\bm{s}}^{\varepsilon_{1}\varepsilon_{2}}$ such that  $\bm{s}\in R_{1}$. The other is $\C_{R_{3}}$, and consists of terms $\A_{\bm{s}}^{\varepsilon_{1}\varepsilon_{2}}$ such that either $\bm{s}-\varepsilon_{1}\Lambda_{1}\in R_{3}$ or $(\bm{s}\in R_{4}, \varepsilon_{2}=1$ and $\bm{s}-\varepsilon_{1}\Lambda_{1}-\Lambda_{2}\in R_{3})$. After quotienting by these acyclic subcomplexes, define two further acyclic quotient complexes $C_{R_{2}}$ consisting of $\A_{\bm{s}}^{\varepsilon_{1}\varepsilon_{2}}$ with $\bm{s}\in R_{2}$, and $\C_{R_{4}}$ consisting of $\A_{\bm{s}}^{\varepsilon_{1}\varepsilon_{2}}$ such that $\bm{s}-\varepsilon_{2}\Lambda_{2}\in R_{4}$. Let $\C_{Q}$ be the resulting subcomplex which is shown as in Figure \ref{mix}. The case where $p_{1}<0, p_{2}>0$ is similar.

\begin{figure}[H]
\centering
\includegraphics[width=\textwidth]{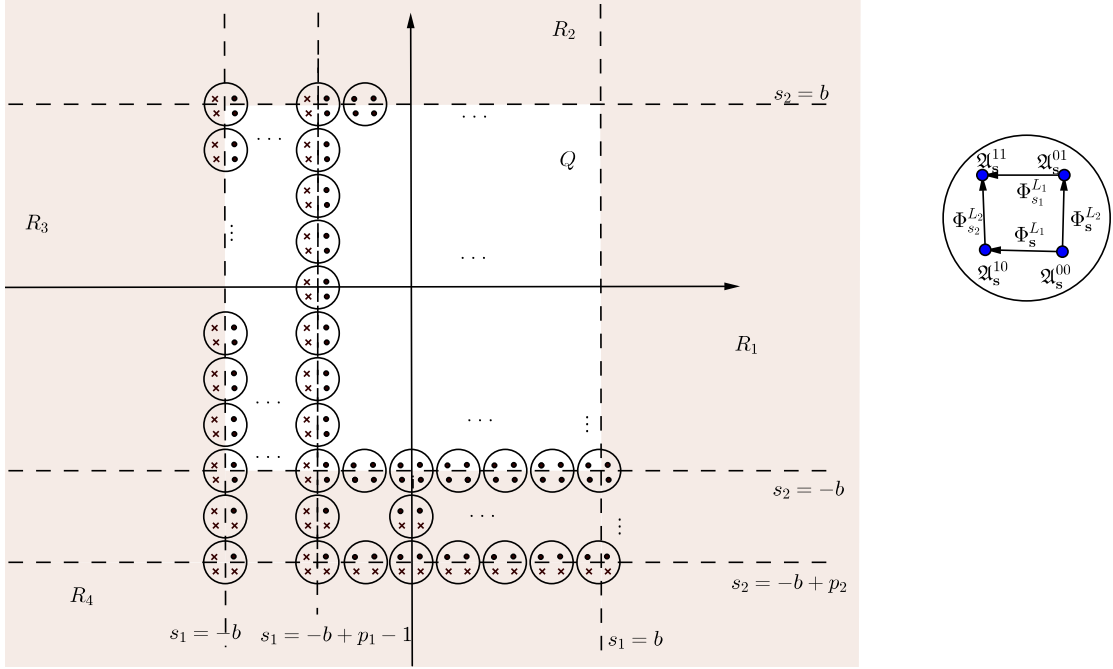}
\caption{Truncated complex for $p_1>0, p_2<0$.}
\label{mix}
\end{figure}

The truncated complex $\C_{Q}$ with the differential obtained by restricting $D$ to $\C_{Q}$ is homotopy equivalent to $(\C(\bH^{\L}, \Lambda), D)$. Then the homology of the truncated complex is isomorphic to $HF^{-}(S^{3}_{p_{1}, p_{2}}(\L))$ up to some grading shift which is independent of the link, but only depends on the homological data \cite{ManOzs}.

For L--space links, Y. Liu introduced the \emph{perturbed surgery formula} to compute the homology of the truncated complex. For the rest of the subsection, we let $\L=L_{1}\cup L_{2}$ denote a 2-component L--space link with vanishing linking number. By Theorem \ref{l-space link cond}, each sublink is also an L--space link. Then 
$$H_{\ast}(\A^{-}(\bH^{L-M}, \psi^{M}(\bm{s})))\cong \F[[U]]$$
 for all $\bm{s}\in \H(L)$ and all sublinks $M\subset L$ \cite{LiuY1, LiuY2}. Moreover, since 
 $\A^{-}(\bH^{L-M}, \psi^{M}(\bm{s}))$ is defined as a bounded complex of free finitely generated $\F[[U]]$--modules, and its homology is also free, $\A^{-}(\bH^{L-M}, \psi^{M}(\bm{s}))$  is homotopy equivalent to $\F[[U]]$.
 
Therefore the surgery complex is homotopy equivalent to the  \emph{perturbed surgery complex} where each $\A^{-}(\bH^{L-M}, \psi^{M}(\bm{s}))$ is replaced by $\F[[U]]$  with the zero differential. The maps $\Phi^{\overrightarrow{L_{i}}}_{\psi^{M}(\bm{s})}$ are replaced as follows:
 $$\tilde{\Phi}^{\pm L_{i}}_{\bm{s}}=U^{H(\pm s_{1}, \pm s_{2})-H_{\bar{i}}(\pm s_{\bar{i}})}: \F[[U]]\rightarrow \F[[U]], $$
 $$\tilde{\Phi}^{\pm L_{i}}_{s_{i}}=U^{ H_{i}(\pm s_{i})}: \F[[U]]\rightarrow \F[[U]]. $$
 Here $\bar{i}\in \lbrace 1, 2 \rbrace \setminus \lbrace i \rbrace$ and $H_{i}(s_{i})$ denotes the $H$-function for $L_{i}$, $i=1, 2$. Finally,  the ``higher'' differentials $\Phi^{\pm L_{1}\pm L_{2}}_{\bm{s}}$ are replaced by some
 differentials $\tilde{\Phi}^{\pm L_{1}\pm L_{2}}_{\bm{s}}$ which must vanish by parity reasons \cite[Lemma 5.6]{LiuY2}.
 
 We will denote the resulting perturbed truncated complex by $(\widetilde{\C_Q}, D)$.
Its homology is isomorphic to the Heegaard Floer  homology of  $S^3_{\p}(\L)$ \cite{ManOzs, LiuY2}. 
Because we are using truncated complexes from here on, it suffices to consider polynomials over $\F[U]$. 
\begin{remark}
Similar complexes and their truncations can be defined for any link with an arbitrary number of components and vanishing pairwise linking numbers. However, for general links with two components the higher differentials could be nontrivial. For L--space links with three or more components one can define the perturbed complex as above, 
but the higher differentials might survive in it as well. See also \cite{Lidman} for a discussion of associated spectral sequences.
\end{remark}

\subsection{Gradings}
\label{subsec:gradings}

In the above discussion we ignored the gradings on all the complexes involved in the surgery formula. 
The homological grading on the surgery complex consists of three separate parts: 
\begin{itemize}
\item[(a)]  The {\em Maslov grading} on $\A^{-}(\bH^{L-M}, \psi^{M}(\bm{s}))$ as a subcomplex of $\A^-(S^3)$.
\item[(b)]  The {\em shift} depending on $\bm{s}$ but not on $M$ (see Remark \ref{rem:shift}). 
\item[(c)]  The {\em cube degree} which we define as 2 for $\A^{00}_{\bm{s}}$, 1 for $\A^{01}_{\bm{s}}$ and $\A^{10}_{\bm{s}}$ and 0 for $\A^{11}_{\bm{s}}$. 
\end{itemize}
We will call the {\em internal degree} the sum of the first two parts and denote it by $\deg$. The homological degree is then the sum of the internal degree and the cube degree. The components of the differential in the surgery complex  change these degrees differently: $\Phi_{\bm{s}}^{\emptyset}$ decreases the internal degree by 1 and preserves the cube degree, 
$\Phi_{\bm{s}}^{\pm L_i}$ preserve the internal degree and  decrease the cube degree by 1, and $\Phi_{\bm{s}}^{\pm L_1\pm L_2}$ increase the internal degree by 1 and  decrease the cube degree by 2. 
The action of $U$ decreases the internal degree by 2 and preserves the cube degree. 
Note that after perturbation of the surgery complex, the only non-vanishing differentials are the $\Phi_{\bm{s}}^{\pm L_i}$, which preserve the internal degree.

\begin{remark}
Note that our cube degrees (shifts applied to cells) differ from the shifts in the Ozsv\'ath-Szab\'o mapping cone formula. In is important to note that the calculations we do with our surgery complex are not absolutely graded, but depend on an overall shift calculated from a two-component unlink. 
\end{remark}

For L--space links, we can replace $\A^{-}(\bH^{L-M}, \psi^{M}(\bm{s}))$ by a copy of $\F[U]$, the internal degree in it is completely determined by the internal degree of the generator. 
For $M\subseteq \{1,2\}$ let $z_M(\bm{s})$ denote the 
generator in the homology of $\mathfrak{A}^{-}(\mathcal{H}^{L-M}, \psi^{M}(\bm{s}))$. 
By the above, in the perturbed surgery complex the differential preserves the internal degree and  decreases the cube degree by 1.

\begin{proposition}
\label{gradingchange}
The internal degrees of $z_{M}(\bm{s})$ can be expressed via the internal degrees of $z_{1,2}(\bm{s})$ as following:
\begin{equation}
\label{deg z1 and z2}
\deg z_{1}(\bm{s})=\deg z_{1,2}(\bm{s}) - 2H_2(s_2),\ \deg z_{2}(\bm{s})=\deg z_{1,2}(\bm{s}) - 2H_1(s_1),
\end{equation}
\begin{equation}
\label{deg z0}
\deg z_{\emptyset}(\bm{s})=\deg z_{1,2}(\bm{s}) - 2H(s_1,s_2).
\end{equation}
Also, the internal degrees of $z_{1,2}(\bm{s})$ satisfy the following recursive relations:
\begin{equation}
\label{grading change 1}
\deg z_{1,2}(s_1+p_1,s_2)=\deg z_{1,2}(s_1,s_2)+2s_1,
\end{equation}
\begin{equation}
\label{grading change 2}
\deg z_{1,2}(s_1,s_2+p_2)=\deg z_{1,2}(s_1,s_2)+2s_2.
\end{equation}
\end{proposition}

\begin{remark}
\label{rem:shift}
The {\em shift} mentioned in the beginning of this subsection is nothing but $\deg z_{1,2}(\bm{s})$.
\end{remark}

\begin{proof}
The differential has the following form:
\begin{eqnarray*}
	D (z_{\emptyset}(\bm{s}))&=&U^{H(\bm{s})-H_1(s_1)}z_{2}(s_1,s_2)+U^{H(\bm{s})-H_2(s_2)}z_{1}(s_1,s_2)+ \\
 && U^{H(-\bm{s})-H_1(-s_1)}z_{2}(s_1,s_2+p_2)+U^{H(-\bm{s})-H_2(-s_1)}z_{1}(s_1+p_1,s_2), \\
 	D (z_{2}(s_1,s_2)) &=& U^{H_1(s_1)}z_{1,2}(s_1,s_2)+U^{H_1(-s_1)}z_{1,2}(s_1+p_1,s_2), \\
	D (z_{1}(s_1,s_2)) &=& U^{H_2(s_2)}z_{1,2}(s_1,s_2)+U^{H_2(-s_2)}z_{1,2}(s_1,s_2+p_2),\\
	D (z_{1,2}(s_1,s_2)) &=&0.
\end{eqnarray*}
The differential preserves the internal degree, therefore $\deg z_{1}(s)=\deg z_{1,2}(s)-2H_2(s_2)$ and $\deg z_{\emptyset}(s)=\deg z_{1}(s)-2(H(\bm{s})-H_2(s_2))$.
By Lemma \ref{h-function symmetry},  $H_1(-s_1)=H_1(s_1)+s_1$, $H_2(-s_2)=H_2(s_2)+s_2$. Therefore
\[
-2H_1(s_1)+\deg z_{1,2}(s_1,s_2)=-2H_1(-s_1)+\deg z_{1,2}(s_1+p_1,s_2)=
\]
\[
-2H_1(s_1)-2s_1+\deg z_{1,2}(s_1+p_1,s_2),
\]
which implies \eqref{grading change 1} and \eqref{grading change 2}.
\end{proof}

\subsection{Associated CW complex}
\label{subsec:CW}

Observe from the definition of the iterated cone, we may assign each summand of $\C_Q$ with the cells of a quotient or sub-complex $\CW(\bm{p}, \bm{i}, \bm{b})$ of 
a finite rectangular CW complex $R$, in a similar manner as was done for knots. In particular, each $\A^{00}_{\bm{s}}$ corresponds to a 2-cell, each of $\A^{01}_{\bm{s}}$ and $\A^{10}_{\bm{s}}$ to a 1-cell, and $\A^{11}_{\bm{s}}$ to a 0-cell, with boundary maps specified by \eqref{maps}. For example, the following diagram shows the 2-cell corresponding with $\A^{00}_{\bm{s}}$ when $p_1, p_2>0$. 
\begin{equation}
\xymatrix@C=30pt@R=30pt{
\A^{11}_{\bm{s}+\Lambda_2} & \A^{01}_{\bm{s}+\Lambda_2} \ar[l]_{\Phi^{L_1}} \ar[r]^{\Phi^{-L_1}} & \A^{11}_{\bm{s}+\Lambda_1 +\Lambda_2} \\
\A^{10}_{\bm{s}} \ar[d]^{\Phi^{L_2}} \ar[u]_{\Phi^{-L_2}}& \A^{00}_{\bm{s}} \ar[d]^{\Phi^{L_2}} \ar[u]_{\Phi^{-L_2}}\ar[l]_{\Phi^{L_1}}\ar[r]^{\Phi^{-L_1}} & \A^{10}_{\bm{s}+\Lambda_1} \ar[d]^{\Phi^{L_2}} \ar[u]_{\Phi^{-L_2}}\\
\A^{11}_{\bm{s}} & \A^{01}_{\bm{s}}  \ar[l]_{\Phi^{L_1}}\ar[r]^{\Phi^{-L_1}}& \A^{11}_{\bm{s}+\Lambda_1} \\
}
\end{equation}

In all of the cases of the truncation, the resulting  complex $\CW(\bm{p}, \bm{i}, \bm{b})$ will be a rectangle  on a square lattice, possibly with some parts of the boundary erased.
The squares, edges and vertices are all cells in this complex. 

\begin{figure}[H]
\centering
\includegraphics[width=\textwidth]{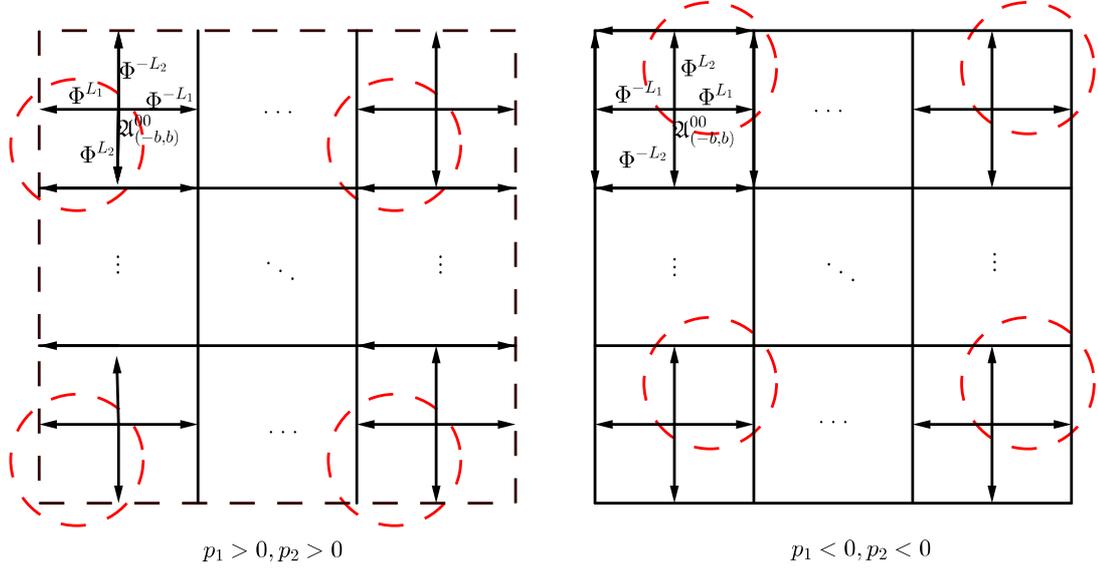}
\caption{Cases (a) and (b).}
\label{surgery 1}
\end{figure}

We can consider the corresponding chain complex $C$ over $\F$ generated by these cells
and the usual differential $\partial$. The homology of this complex is naturally isomorphic to the homology of $R$ relative to the union of erased cells. Specifically, we will consider three situations:

\begin{figure}[H]
\centering
\includegraphics[width=2.5in]{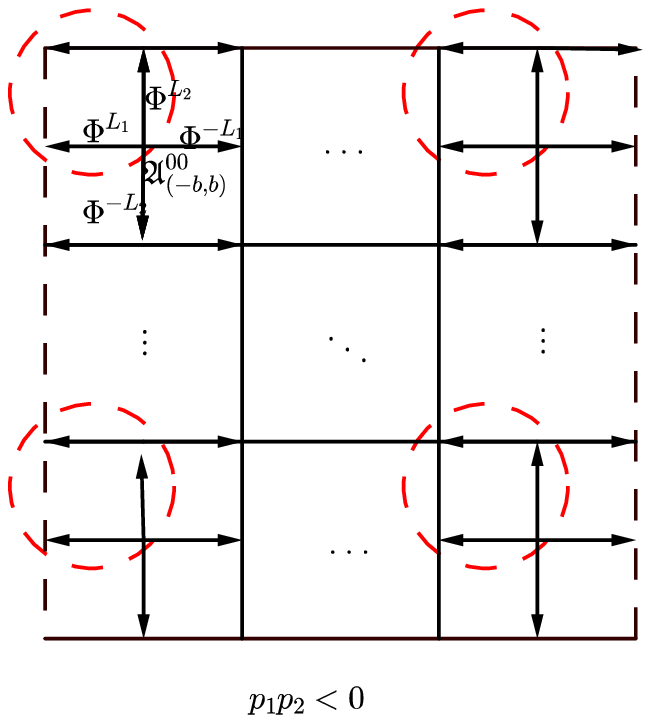}
\caption{Case (c). }
\label{surgery 2}
\end{figure}

\begin{itemize}
\item[(a)] If none of the cells are erased, then $R$ is contractible, so $H_0(C,\partial)\cong \F$ is generated by the class of a 0-cell, and all other homologies vanish. This corresponds to the case when both surgery coefficients are  negative as in Figure \ref{surgery 1}.
\item[(b)] If all 1- and 0-cells on the boundary of $R$ are erased, then $(R,\partial R)\simeq (S^2,pt)$. Therefore $H_2(C,\partial)\cong\F$ is generated by the sum of all 2-cells, and all other homologies vanish. This corresponds to the case when both surgery coefficients are  positive.
\item[(c)] If all 1- and 0-cells on a pair of opposite sides of $R$ are erased, then $R$ relative to erased cells is homotopy eqivalent to $(S^1,pt)$. Therefore $H_1(C,\partial)\cong \F$ is generated by the class of any path connecting opposite erased boundaries, and all other homologies vanish. This corresponds to the case when the surgery coefficients have different signs as in Figure \ref{surgery 2}.
\end{itemize}


\section{The $d$-invariant of surgery}
\label{subsec:dfromcells}

\subsection{$d$-invariant from cells}
Given the CW complex $\CW(\bm{p},\bm{i}, \bm{b})$ in Section \ref{subsec:CW}, we can reconstruct the (perturbed, truncated) surgery complex $(\widetilde{\C_Q},D)$  as follows. Each cell $\sq$ of  $\CW(\bm{p},\bm{i}, \bm{b})$ corresponds to a copy of $\F[U]$ generated by some element $z(\sq)$.
It has some internal degree which we will denote by $\deg(\sq)$. Every component of the boundary map in 
$\CW(\bm{p},\bm{i}, \bm{b})$ corresponds to a  component of $D$. By \cite{LiuY2}, $D$ is nonzero and hence given by multiplication by a certain power of $U$.  
By Proposition \ref{gradingchange} the internal degrees $\deg(\sq)$ have the same parity and $\deg(\sq_i)\geq \deg(\sq)$ if $\sq_i$ shows up in the differential of $\sq$. We get the following equation:
\begin{equation}
D(z(\sq))=\sum U^{\frac{1}{2}(\deg(\sq_i)-\deg(\sq))}z(\sq_i),\ \text{if}\ \partial \sq=\sum \sq_i.
\end{equation}
As above, the complex $(\widetilde{\C_Q},D)$  is bigraded: the {\em cube grading} of $z(\sq)U^{k}$ equals the dimension of $\sq$, while the internal degree of $z(\sq)U^k$ equals $\deg(\sq)-2k$. The differential $D$ preserves the internal degree and decreases the cube grading by 1. The actual homological grading on the surgery complex is the sum of two degrees.

The homology of $(\widetilde{\C_Q},D)$  could be rather complicated,  and is similar to the so-called lattice homology considered by N\'emethi \cite{Nemethi}. Nevertheless, the homology of  $(\widetilde{\C_Q},D)$   modulo $U$-torsion can be described explicitly. Let $(C,\partial)$ denote the chain complex computing the cellular homology of $CW(\bm{p},\bm{i}, \bm{b})$. Consider the map $$\varepsilon: \widetilde{\C_Q}\to C, \quad  \varepsilon(z(\sq)U^k)=\sq.$$ Clearly, $\varepsilon$ is a chain map, that is, $\partial \varepsilon=\varepsilon D$. Given a cell $\sq$, we call $z(\sq)U^k$ its graded lift of internal degree $\deg(\sq)-2k$. The following proposition is straightforward.

\begin{proposition}
\label{graded lift unique}
Let $c$ be a chain in $C$. It admits a graded lift of internal degree $N$ (that is, a homogeneous chain $\alpha$ in 
$\widetilde{\C_Q}$ such that $\varepsilon(\alpha)=c$) if and only if $N$ is less than or equal to the minimal internal degree of cells in $c$. If a graded lift exists, it is unique. Any two graded lifts of different internal degrees are related by a factor $U^k$ for some $k$.
\end{proposition}

\begin{lemma}
Let $z$ be a homogeneous chain in $\widetilde{\C_Q}$. Then $z$ is a cycle if and only if $\varepsilon(z)$ is a cycle. Also, $U^{k}z$ is a boundary for large enough $k$ if and only if $\varepsilon(z)$ is a boundary.
\end{lemma}

\begin{proof}
If $z$ is a cycle then  $\varepsilon(z)$ is a cycle since $\varepsilon$ is a chain map. Conversely, if $\varepsilon(z)$ is a cycle, then $\varepsilon(D(z))=0$, and hence $D(z)=0$.

If $U^{k}z=D \alpha$ then by applying $\varepsilon$ we get $\varepsilon(z)=\partial \varepsilon(\alpha)$. Conversely, assume that $\varepsilon(z)=\partial \beta$. Pick a graded lift $\alpha$ of internal degree $N$ such that  
$\varepsilon(\alpha)=\beta$. Then $\varepsilon(D \alpha)=\varepsilon(z)$, so $D\alpha$ is a graded lift of $z$. By Proposition \ref{graded lift unique} we have $D \alpha=U^{\frac{1}{2}(\deg(z)-N)}z$.
\end{proof}

\begin{corollary}
The free part of the homology $H_*(\widetilde{\C_Q},D)/Tors$ is generated by the graded lifts of representatives of homology classes in $H_*(C,\partial)$. Two classes are equivalent if and only if they have the same internal degree and lift the same homology class.
\end{corollary}

It follows that in all cases (a)-(c) in section \ref{subsec:CW} the free part $H_*(\widetilde{\C_Q},D)/Tors$ is isomorphic to $\F[U]$.
Let $d$ denote the internal degree of the generator of this copy of $\F[U]$ (this is essentially the $d$-invariant of the surgery).
We are ready to compute $d$:

\begin{theorem}
\label{square erase}
The $d$-invariant of the complex $(\widetilde{\C_Q},D)$ can be computed in terms of $\CW(\bm{p},\bm{i}, \bm{b})$ as following:
\begin{itemize}
\item[(a)] If no cells of the rectangle $R$ are erased, this is the maximal value of $\deg(\sq)$ for 0-cells $\sq$.
\item[(b)] If all boundary cells are erased, this is the minimal value of $\deg(\sq)$ for 2-cells $\sq$.
\item[(c)] If two sides are erased, this is $\max_{c}\min_{\sq\in c}\deg(\sq)$, where 
$c$ is a simple lattice path connecting the erased sides.
\end{itemize}
\end{theorem}

\begin{proof}
In (a), $H_*(C,\partial)$ is generated by the class of a point (that is, a 0-cell). All points are equivalent in $\widetilde{\C_Q}$ modulo torsion, and any lift of a 0-cell $\sq$ has the form $U^{k}z(\sq)$ and has internal degree less than or equal to $\deg(\sq)$. Therefore the maximal internal degree of a graded lift of a point equals $\max \deg(\sq)$. 

In (b), $H_*(C,\partial)$ is generated by the sum of all 2-cells. The graded lift of this chain exists in internal degrees $\min \deg(\sq)$ and less.

In (c), similarly, for a given 1-chain $c$ representing the nontrivial homology class, a graded lift is possible in internal degrees 
$\min_{\sq\in c}\deg(\sq)$ and less. Therefore to find the internal degree of the generator of $\F[U]$ we need to take the maximum over all $c$. It remains to notice that any such $c$ contains a simple lattice path $c'$ connecting the erased sides, and $\min_{\sq\in c'}\deg(\sq)\ge \min_{\sq\in c}\deg(\sq)$.
\end{proof}

\subsection{Proof of Theorem \ref{thm:generalizedniwu}}
\label{d-grading}
Let us describe the gradings on the surgery complex in more detail.  

Let us fix a $\spinc$--structure $\bm{i}=(i_1, i_2)$ on $S^3_{\p}(\L)$. 
The four quadrants on the plane are denoted $(\pm,\pm)$. 
In each quadrant, we can find a unique point $s_{\pm \pm}(\bm{i})$ in $\spinc$--structure $\bm{i}$ that is the closest to the origin, as in Figure \ref{fourpoints}. 
If $i_1=0$ or $i_2=0$ then some of $s_{\pm\pm}$ coincide, and in particular, if $i_1=i_2=0$ then $s_{\pm \pm}(\bm{i})=(0,0)$ for all signs. We also define integers $s_{\pm}^{(1)}$ and $s_{\pm}^{(2)}$ to be the coordinates of the points, i.e. 
\[
s_{\pm \pm}=(s_{\pm}^{(1)},s_{\pm}^{(2)}).
\]

\begin{lemma}
\label{grafor}
If $p_1>0, p_2>0$, then 
$$
\deg z_{\emptyset}(s_{\pm\pm}(\bm{i}))=\deg z_{1, 2}(s_{++}(i_1, i_2))-2h(s_{\pm \pm}(i_1, i_2)).
$$
\end{lemma}

\begin{proof}
Assume that $s_{++}(i_1, i_2)=(s_1, s_2)$. By Equation \ref{deg z0}, 
$$\deg z_{\emptyset}(s_{++}(\bm{i}))=\deg z_{1, 2}(s_{++}(\bm{i}))-2H(s_{++}(\bm{i})).$$ 
Suppose $s_1\neq 0, s_2\neq 0$. By Proposition \ref{gradingchange}, 
\begin{multline*}\deg z_{\emptyset}(s_{-+}(\bm{i}))=\deg z_{1, 2}(s_{-+}(\bm{i}))-2H(s_{-+}(\bm{i}))=\\
\deg z_{1, 2}(s_{++}(\bm{i}))-2(s_1-p_1)-2H(s_{-+}(\bm{i})).
\end{multline*}
Similarly,
$$\deg z_{\emptyset}(s_{+-}(\bm{i}))=\deg z_{1, 2}(s_{++}(\bm{i}))-2(s_2-p_2)-2H(s_{+-}(\bm{i})),$$
$$\deg z_{\emptyset}(s_{--}(\bm{i}))=\deg z_{1, 2}(s_{++}(\bm{i}))-2(s_1-p_1)-2(s_2-p_2)-2H(s_{--}(\bm{i})).$$
For the unlink $O$ with two components, we have
$$H_{O}(s_{++}(\bm{i}))=0, H_{O}(s_{-+}(\bm{i}))=p_1-s_1, H_{O}(s_{+-}(\bm{i}))=p_2-s_2$$
and 
$$H_{O}(s_{--}(\bm{i}))=p_1-s_1+p_2-s_2.$$
Therefore, 
$$\deg z_{\emptyset}(s_{\pm\pm}(\bm{i}))=\deg z_{1, 2}(s_{++}(\bm{i}))-2H(s_{\pm\pm}(\bm{i}))+2H_{O}(s_{\pm\pm}(\bm{i}))=$$
$$\deg z_{1, 2}(s_{++}(\bm{i}))-2h(s_{\pm\pm}(\bm{i})).$$
If $s_1=0$ and $s_2\neq 0$, then 
$$s_{\pm +}(\bm{i})=(0, s_2), \quad s_{\pm -}(\bm{i})=(0, s_2-p_2).$$
It is easy to check that the equation in Lemma \ref{grafor} still holds. Similarly, it also holds in the case $s_2=0$.
\end{proof}

\begin{figure}[H]
\centering
\includegraphics[width=2.6in]{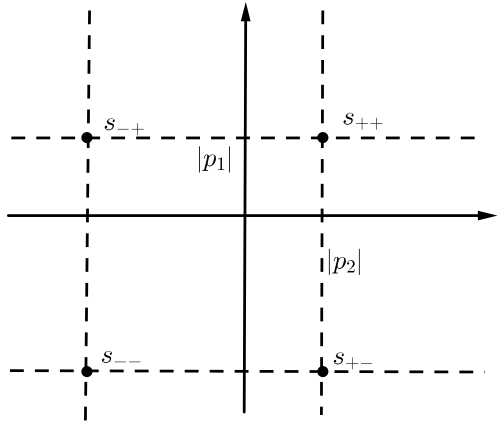}
\caption{For each $\spinc$--structure $\bm{i}$, there is a unique point $s_{\pm \pm}(\bm{i})$ in each quadrant that is the closest to the origin.}
\label{fourpoints}
\end{figure}

\begin{lemma}
\label{grafor1d}
If $p_1>0$ then
$$
\deg z_{2}(s_{\pm}^{(1)},t)=\deg z_{1, 2}(s_{+}^{(1)},t)-2h_1(s_{\pm}^{(1)}).
$$
\end{lemma}

\begin{proof}
The proof is similar to the proof of Lemma \ref{grafor}. Assume that $s_1=s_{+}^{(1)}\neq 0$. Then $s_{-}^{(1)}=s_1-p_1$
and
$$
\deg z_{2}(s_1,t)=\deg z_{1, 2}(s_1,t)-2H_1(s_1)=\deg z_{1, 2}(s_1,t)-2h_1(s_1), 
$$
$$
\deg z_{2}(s_1-p_1,t)=\deg z_{1, 2}(s_1,t)-2H_1(s_1-p_1)-2(s_1-p_1)=\deg z_{1, 2}(s_1,t)-2h_1(s_1-p_1).
$$
\end{proof}

\noindent
{\bf Proof of Theorem \ref{thm:generalizedniwu}:}
{\bf (a)} Assume $p_1,p_2<0$. Then by Theorem \ref{square erase}(a), in which case no cells are erased, we get 
\[
	d(S^3_{\p}(\L),(i_1,i_2))=\max_{s_k=i_k+a_kp_k}\deg z_{1,2}(s_1,s_2).
\]
The internal degree of $z_{1, 2}(s_1, s_2)$ does not depend on the link, but depends on the framing matrix $\Lambda$. Since the $(p_1, p_2)$-surgery on the unlink decomposes as $L(p_1, 1)\#L(p_2, 1)$ and has the same framing matrix, then 
\[
	d(S^{3}_{\p}(\L), (i_1, i_2))=\phi(p_1, i_1)+\phi(p_2, i_2).
\]

{\bf (b)} Assume $p_1,p_2>0$. Then by Theorem \ref{square erase}(b), in which case all boundary cells are erased, we get 
\[
d(S^3_{\p}(\L),(i_1,i_2))=\min_{s_k=i_k+a_kp_k}\deg z_{\emptyset}(s_1,s_2)+2.
\]
Note that we add 2 here because the homological degree of a generator is a sum of $\deg$ and its cube degree.
Let us prove that $\deg z_{\emptyset}(s_1,s_2)$ decreases towards the origin.  Indeed, by combining \eqref{deg z0} and \eqref{grading change 1}, we get:
\begin{equation*}
\label{grading change z0}
\deg z_{\emptyset}(s_1+p_1,s_2)=\deg z_{\emptyset}(s_1,s_2)+2s_1+2H(s_1,s_2)-2H(s_1+p_1,s_2).
\end{equation*}
By Lemma \ref{h-function increase}
\[
0\le H(s_1,s_2)-H(s_1+p_1,s_2)\le p_1.
\]
Therefore for $s_1\ge 0$ we have $\deg z_{\emptyset}(s_1+p_1,s_2)\ge \deg z_{\emptyset}(s_1,s_2)$ and
for $s_1\le -p_1$ we have $\deg z_{\emptyset}(s_1+p_1,s_2)\le \deg z_{\emptyset}(s_1,s_2)$.

Therefore the minimal value is achieved at one of $s_{\pm \pm}(\bm{i})$. 
By Lemma \ref{grafor}, 
$$\deg z_{\emptyset}(s_{\pm\pm}(\bm{i}))=\deg z_{1, 2}(s_{++}(\bm{i}))-2h(s_{\pm \pm}(\bm{i})).$$ 
Then 
$$
d(S^3_{\p}(\L),(i_1,i_2))=\deg z_{1, 2}(s_{++}(\bm{i}))-2\max h(s_{\pm \pm}(\bm{i}))+2,
$$ 
where, as above, $\deg z_{1, 2}(s_{++}(\bm{i}))$ does not depend on the link.
For the unlink  $h=0$, hence
$$\deg z_{1, 2}(s_{++}(\bm{i}))+2=d(S^3_{\p}(O),(i_1,i_2))=\phi(p_1, i_1)+\phi(p_2, i_2).$$

{\bf (c)} Assume that $p_1>0, p_2<0$. Then by Theorem \ref{square erase}(c), we get 
$$
d(S^{3}_{\p}(\L), (i_1, i_2))=\max_{c}\min_{\sq\in c}\deg(\sq)+1
$$
where $c$ is a simple lattice path connecting the erased sides. Let $c(t)$ be the horizontal path connecting erased boundaries at height $t$. 
Let us compute $\min_{\sq\in c(t)}\deg(\sq)$. By 
Proposition \ref{gradingchange} we get 
$$
\deg z_{2}(s_1+p_1, t)=\deg z_{2}(s_1, t)+2H_1(s_1)-2H_1(s_1+p_1)+2s_1.
$$
and similarly to case (b) we conclude that the minimum is achieved at one of $(s_{\pm}^{(1)},t)$.
Also, by Lemma \ref{grafor1d} we get
\begin{equation}
\label{grading121}
\min_{\sq\in c(t)}\deg(\sq)=\deg z_{1,2}(s_{+}^{(1)}, t)-2\max h_1(s_{\pm}^{(1)}).
\end{equation}

By  Proposition \ref{gradingchange}, we have
$$
\deg z_{2}(s_1, s_2+p_2)=\deg z_{2}(s_1, s_2)+2s_2.
$$
Since $p_2<0$, this means that for fixed $s_1$ the internal degree of $z_{2}(s_1, t)$ increases towards the origin and achieves its maximum at $t_0=s_{+}^{(2)}+p_2$. 

For an arbitrary simple path $c'$ connecting the erased boundaries, it must contain a horizontal segment corresponding to  
$z_2(s_{\pm}^{(1)}, t)$.
Then 
$$
\min_{\sq\in c'}\deg(\sq)\le \deg z_2(s_{\pm}^{(1)}, t)\le \deg z_2(s_{\pm}^{(1)}, t_0)=\min_{\sq\in c(t_0)}\deg(\sq).
$$
Therefore,
$$
\max_{c}\min_{\sq\in c}\deg(\sq)=\min_{\sq\in c(t_0)}\deg(\sq)=\deg z_{1,2}(s_{+}^{(1)},s_{+}^{(2)}+p_2)-2\max h_1(s_{\pm}^{(1)}).
$$
The second equality follows  the same argument as the one for \eqref{grading121}.
Again, the first term does not depend on the link and hence equals the $d$-invariant of the lens space:
$$
\deg z_{1,2}(s_{+}^{(1)},s_{+}^{(2)}+p_2)+1=d(S^3_{\p}(O),i_1,i_2)=\phi(p_1, i_1)+\phi(p_2, i_2).
$$
Finally, it follows from \cite[Proposition 1.6]{NiWu} that 
$$
d(S^3_{p_1}(L_1),i_1)=\phi(p_1, i_1)-2\max h_1(s_{\pm}^{(1)}),
$$
so
$$
d(S^3_{\p}(\L),(i_1,i_2))=d(S^3_{p_1}(L_1),i_1)+\phi(p_2, i_2).
$$\null\hfill\qedsymbol

\subsection{Example: $d$-invariants and twisting}
\label{subsec:Kirby}

We can use this result to prove a curious property of the $H$-function for L--space links of linking number zero. Suppose that $L_1$ is an unknot. Then after performing a Rolfsen twist, a $(+1,p_2)$-surgery on $\L$ is homeomorphic to $p_2$-surgery on some knot $L'_2$ obtained from $L_2$ by a negative full twist \cite[Section 5]{GS}. See Figure \ref{rolfsen}. Note that while Theorem \ref{l-space link cond} implies that $L_2$ is an L--space knot (since $\L$ is an L--space link),  $L_2'$ does not need to be  an L--space knot, see Corollary \ref{cor: L2prime L space}. 
  
\begin{theorem}
\label{H for slamdunk}
Let $\L=L_1\cup L_2$ be an L--space link of linking number zero, and $L_{1}$ is an unknot. The $H$-function for $L'_2$ equals $H(0,s_2)$.
\end{theorem}

\begin{proof}
By definition, the $H$--function is equal (up to a shift) to the $d$-invariant of $S^3_{p_2}(L'_2)$ or, equivalently, of
$S^3_{1,p_2}(\L)$ for $p_2\gg 0$. Since $p_1=1$, a $\spinc$-structure on the surgery is given by a lattice point $(0,i_2)$ where $-p_2/2\le i_2 \le p_2/2$. 
The $d$-invariant is determined by the values of the $H$-function of $\L$ at the points $(0,i_2)$. By Theorem \ref{thm:generalizedniwu} we get 
$$
d(S^3_{p_2}(L'_2), i_2)=d(S^3_{1,p_2}(L), (0, i_2))=0+\phi(p_2,i_2)-2h(0,i_2).
$$
Indeed, $\phi(1,0)=0$ since $1$-surgery of $S^3$ along the unknot is $S^3$. Then $h(0, i_2)=h_{L'_2}(i_2)$. Hence, the $H$-function for $L'_{2}$ equals $H(0, s_2)$.
\end{proof}

\begin{figure}[H]
\centering
\begin{tikzpicture}
\node[anchor=south west,inner sep=0] at (0,0) {\includegraphics[width=0.4\textwidth]{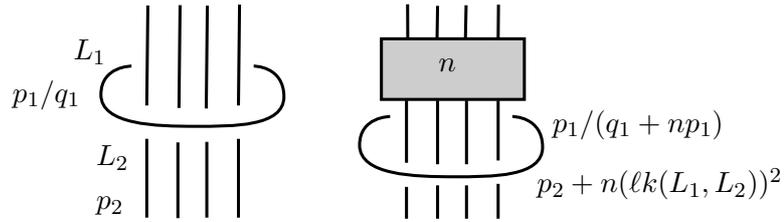}};
\node[label=above right:{$L_1$}] at (-0.6,1.8){};
\node[label=above right:{$p_1/q_1$}] at (-1.4,1.2){};
\node[label=above right:{$p_1/(q_1+np_1)$}] at (5.7,0.8){};
\node[label=above right:{$n$}] at (4.2,1.7){};
\node[label=above right:{$p_2$}] at (-0.3, -0.2){};
\node[label=above right:{$L_2$}] at (-0.3, 0.4){};
\node[label=above right:{$p_2+n(\ell{k}(L_1,L_2))^2$}] at (5.5, 0){};
\end{tikzpicture}
\caption{A Rolfsen twist. Here we take $p_1/q_1=\pm 1$ and $n=\mp1$.}
\label{rolfsen}
\end{figure}

\begin{remark}
Similarly, we can consider $(-1, p_2)$-surgery on $\L$. Let  $L''_2$ be  the knot obtained from $L_2$ by a   positive full twist. By Theorem  \ref{thm:generalizedniwu}, 
$$d(S^{3}_{-1, p_2}(\L), i_2)=d(S^{3}_{p_2}(L''_2), i_2)=d(S^{3}_{p_2}(L_2), i_2).$$
Hence, $H_{L_2}(s)=H_{L''_2}(s)$.
\end{remark}

\begin{example}
If $L$ is the positively-clasped Whitehead link then $L'_2$ is the right-handed trefoil, and $L''_2$ is the figure eight knot. See Figure \ref{whitehead}. The values of the $H$-function for the Whitehead link on the axis agree with the values of the $H$-function of the trefoil (see also Example \ref{wh H}). The values of the $H$-function for the unknot agree with the  values of the $H$-function for the figure eight knot. 
\end{example}

Assume from now on that $L$ is nontrivial so that $H(0,0)>0$. If $L_1$ is an unknot, then by the stabilization property (Lemma \ref{h-function bdy}) for $s_2\gg 0$ we have $H(0,s_2)=H_1(0)=0$.
We define 
\[
b_2=\max\{s_2: H(0,s_2)>0\}.
\]
Clearly, $b_2\ge 0$. Since $H(\bm{s})=h(\bm{s})$ for $\bm{s}\succeq \mathbf{0}$, note that we could have also defined $b_2$ as $\max\{s_2: h(0,s_2)>0\}$.
\begin{corollary}
\label{tau invariant}
In the above notations one has $\nu^{+}(L'_2)=b_2+1$.
\end{corollary}

\begin{proof}
By Theorem \ref{H for slamdunk} $H(0,s_2)$ agrees with the $H$-function of $L'_2$, and following the definition of the invariant $\nu^+$ in \cite{HW},
\[
\nu^{+}(L'_2)=\max\{s_2:H_{L'_2}(s_2)>0\}+1=\max\{s_2:H(0,s_2)>0\}+1=b_2+1.\qedhere
\]
\end{proof}
In particular this means that $L'_2$ has nonzero $H$-function and positive $\nu^{+}$-invariant. Note that Proposition \ref{prop:tau} is the special case of Corollary \ref{tau invariant} when we assume that both $L_1$ and $L_2$ are unknotted.

\subsection{Example: $\pm 1$ surgery}

Let $\L=L_{1}\cup L_{2}$ denote an L--space link with vanishing linking number.  If $p_{1}=p_{2}=-1$, then by Theorem \ref{square erase}, no cells in the truncated square are erased, and the $d$-invariant of the surgery complex $d(S^{3}_{-1, -1}(\L))$ equals  the $d$-invariant of the lens space $L(-1, 1)\# L(-1, 1)$ which is zero.

If $p_{1}=p_{2}=1$, there is a unique $\spinc$-structure $(0, 0)$ on $d(S^{3}_{1, 1}(\L))$. Then $s_{\pm \pm}(0, 0)=(0, 0)$. By Theorem \ref{thm:generalizedniwu},
\begin{equation*}
\label{onesurg}
d(S^{3}_{1, 1}(\L))=-2h(0, 0).
\end{equation*}

\begin{figure}[H]
\centering
\begin{tikzpicture}
\node[anchor=south west,inner sep=0] at (0,0) {\includegraphics[width=0.4\textwidth]{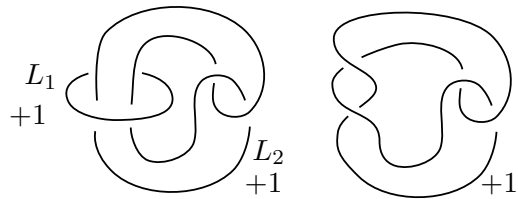}};
\node[label=above right:{$L_1$}] at (-0.8,1.2){};
\node[label=above right:{$+1$}] at (-1,0.7){};
\node[label=above right:{$L_2$}] at (2.2, 0.2){};
\node[label=above right:{$+1$}] at (2.1, -0.2){};
\node[label=above right:{$+1$}] at (5.2, -0.2){};
\end{tikzpicture}
\caption{After $+1$ surgery along component $L_1$ of the positively-clasped Whitehead link we obtain the right-handed trefoil in $S^3$.}
\label{whitehead}
\end{figure}


\section{Classification of L--space surgeries}
\label{unknot components}

For L--space links with unknotted components, we give a complete description of (integral) L--space surgery coefficients.
We define nonnegative integers $b_1,b_2$ as in Corollary \ref{tau invariant}:
$$
b_1=\max\{s_1: h(s_1,0)>0\},\ b_2=\max\{s_2: h(0,s_2)>0\}.
$$

\begin{theorem}
\label{bounds from tau}
Assume that $\L$ is a nontrivial $L$--space link with unknotted components and linking number zero.   Then $S^3_{p_1,p_2}(\L)$ is an L--space if and only if $p_1>2b_1$ and $p_2>2b_2$.
\end{theorem}

\begin{proof}
By Lemma \ref{lem: h increases} we have $h(s_1,s_2)=0$ outside the rectangle $[-b_1,b_1]\times [-b_2,b_2]$.
Also, $h(-b_1,0)=h(b_1,0)>0$, so by Lemma \ref{lem: h increases}, $h(s_1,0)>0$ for $-b_1\le s_1\le b_1$.

Assuming that $p_1>2b_1$ and $p_2>2b_2$, then we can truncate the surgery complex to obtain a rectangle 
where in each $\spinc$ structure  $\bm{i}$, there  is exactly one lattice point $\A^{00}_{\bm{s}}$; see Figure \ref{positive}. Hence, $HF^{-}(S^{3}_{\p}(\L), \bm{i})\cong H_{\ast}(\A^{00}_{\bm{s}})\cong \F[U]$. Therefore $S^{3}_{\p}(\L)$ is an L--space.

Conversely, assume that $S^3_{\p}(\L)$ is an L--space. Let us first prove that $p_1,p_2>0$. Indeed, since $H(0,0)>0$ the boundary of  
$z_{\emptyset}(0,0)$ is divisible by $U$, so let $\alpha=U^{-1}D (z_{\emptyset}(0,0))$. Then $U\alpha$ is a boundary, which implies that $\alpha$ is $U$-torsion in homology. Hence $\alpha$ is $0$ in homology since $S^{3}_{\p}(\L)$ is an L--space. Therefore $\alpha=D(\beta)$ for some $\beta$, and $\beta$ must be supported on all 2-cells outside $(0,0)$. This is possible only if all cells on the boundary are erased, which occurs when $p_1,p_2>0$.

Now, assume that $p_2>0$ and $0<p_1\le 2b_1$. Then $h(-b_1,0)>0$ and $h(p_1-b_1,0)>0$. Similarly, the boundary of  
$z_{\emptyset}(-b_1,0)$ is divisible by $U$, so let $\alpha'=U^{-1} D(z_{\emptyset}(-b_1,0))$ and $\alpha'=D(\beta')$. Then $\deg \beta'=\deg \alpha'=\deg  z_{\emptyset}(-b_1,0)+2$ and $\beta'$ is supported on all 2-cells outside $(-b_1,0)$. In particular, it is supported at $(p_1-b_1,0)$ hence 
$$
\deg z_{\emptyset}(p_1-b_1,0)\ge \deg \beta'= \deg  z_{\emptyset}(-b_1,0)+2.
$$
By swapping the roles of $(-b_1,0)$ and $(p_1-b_1,0)$, we obtain
$$
\deg z_{\emptyset}(-b_1,0)\ge \deg z_{\emptyset}(p_1-b_1,0)+2,
$$
which is a contradiction. Therefore $p_1>2b_1$ and likewise $p_2>2b_2$.
\end{proof}

\begin{remark}
After combining Theorem \ref{bounds from tau} with Corollary \ref{tau invariant}, we obtain the statement of Theorem \ref{thm:taubound} stated in the introduction.
\end{remark}

\begin{example}
For the Whitehead link we have $b_1=b_2=0$, so $S^3_{p_1,p_2}(\L)$ is an L--space  if and only if 
$p_1,p_2>0$. See also \cite{LiuY1} for a detailed discussion of Heegaard Floer homology for surgeries on the Whitehead link.
\end{example}

\begin{example}
It is known \cite{LiuY2} that for $k>0$ the two-bridge link $b(4k^2+4k,-2k-1)$ is an L--space link with linking number zero. 
The corresponding $h$-function was computed in \cite{LiuY2,BG} (see also \cite[Example 4.1]{Liu}), and it is easy to see that $b_1=b_2=k-1$.
Therefore a $(p_1,p_2)$--surgery on $b(4k^2+4k,-2k-1)$ is an L--space if and only if $p_1,p_2>2k-2$.
\end{example}

For more general L--space links with linking number zero, we know that $H(0,0)\ge H_1(0)$ and $H(0,0)\ge H_2(0)$. 
If both of these inequalities are strict, then similarly to the proof of Theorem \ref{bounds from tau} 
one can prove that for L--space surgeries we must have $p_1,p_2>0$. In general, we have the following weaker results. 

\begin{proposition}
\label{not both negative}
Suppose that $\L$ is a nontrivial L--space link with linking number zero.
If $S^3_{p_1,p_2}(\L)$ is an L--space then either $p_1>0$ or $p_2>0$.
\end{proposition}

\begin{proof}
If both $L_1$ and $L_2$ are unknots then the statement follows from Theorem \ref{bounds from tau}. Otherwise assume that $L_1$ is a nontrivial L--space knot, and so $H_1(0)>0$. Assume that both $p_1$ and $p_2$ are negative and $S^3_{p_1,p_2}(\L)$ is an L--space.

Let us choose $s_2$ such that $z_{2}(0,s_2)$ has maximal possible grading. We have 
$$
D(z_{2}(0,s_2))=U^{H_1(0)}(z_{1,2}(0,s_2)+z_{1,2}(p_1,s_2)).
$$
Since $p_1,p_2<0$, then by Theorem \ref{square erase} $z_{1,2}(0,s_2)$ and $z_{1,2}(p_1,s_2)$ are nonzero (and even non-torsion) in homology. They have the same degree, so their sum must vanish. This means that there exists a 1-chain $\gamma$ with endpoints at $(0,s_2)$ and $(p_1,s_2)$ such that its graded lift is bounded by $z_{1,2}(0,s_2)+z_{1,2}(p_1,s_2)$.

Such $\gamma$ must contain a segment connecting $(0,s'_2)$ and $(p_1,s'_2)$ for some $s'_2$, so its graded lift contains
$U^{k}z_{1}(0,s'_2)$ for some $k\ge 0$. Then
\begin{eqnarray*}
\deg z_{1}(0,s'_2)\ge \deg U^{k}z_{1}(0,s'_2)&=&\deg(z_{1,2}(0,s_2)+z_{1,2}(p_1,s_2))\\
 &>& \deg z_{1,2}(0, s_2) - 2H_1(0) = \deg z_{1}(0,s_2).
\end{eqnarray*}
Contradiction, since $z_{1}(0,s_2)$ had maximal possible grading. 
\end{proof}

\begin{proposition}
\label{lem:at least one}
Suppose that $\L$ is an L--space link with linking number zero.
If $S^3_{p_1,p_2}(\L)$ is an L--space then either $S^3_{p_1}(L_1)$ or $S^3_{p_2}(L_2)$ is an L--space.
\end{proposition}

\begin{proof}
If $L_1$ or $L_2$ are unknots, the statement is clear. Suppose that both $L_1$ and $L_2$ are nontrivial with genera $g_1$ and $g_2$. Then we need to prove that either $p_1\ge 2g_1-1$ or $p_2\ge 2g_2-1$. Assume that, on the contrary, $p_1\le 2g_1-2$ and $p_2\le  2g_2-2$. 

Consider the generator $z_{1,2}(s_1,s_2)$. 
It appears in the boundary of $z_{1}(s_1,s_2)$ with coefficient $U^{H_2(s_2)}$, in the boundary of $z_{2}(s_1,s_2)$ with coefficient $U^{H_1(s_1)}$, in the boundary of $z_{1}(s_1, s_2-p_{2})$ with coefficient $U^{H_2(p_2-s_2)}$
and in the boundary  of $z_{2}(s_{1}-p_{1}, s_2)$ with coefficient $U^{H_1(p_1-s_1)}$. 
For $s_1=g_1-1,s_2=g_2-1$, by the assumptions we have $p_1-s_1\le g_1-1$ and $p_2-s_2\le g_2-1$.
Recall that for an L--space knot, 
\[
	g(K)=\nu^{+}(K)=\max\{s:H_{K}(s)>0\}+1.
\] 
Thus, since $L_1$ and $L_2$ are L--space knots, all four exponents $H_1(s_1),H_2(s_2),H_1(p_1-s_1),H_2(p_2-s_2)$ are strictly positive. 
Therefore the cycle $z_{1,2}(s_1,s_2)$ does not appear in the boundary of any chain and hence is nontrivial in homology. On the other hand, by Lemma \ref{not both negative} either $p_1$ or $p_2$ is positive, so by Theorem \ref{square erase} $z_{1,2}(s_1,s_2)$ is a torsion class. Therefore $z_{1,2}(s_1,s_2)$ is a nontrivial torsion class, and $S^3_{p_1,p_2}(\L)$ is not an L--space. Contradiction.
\end{proof}

\begin{remark}
The examples considered in \cite{GN:set,Sarah} show that for many L--space links it is possible to have L--space surgeries with $p_1>0$ and $p_2<0$. For 2-component L--space links with linking number zero, this is not possible (see \cite{Liu19}). For general 2-component L--space links, there are similar results to the ones in   Propositions \ref{not both negative} and \ref{lem:at least one} \cite{Liu19}. 
\end{remark}


\section{Relationship with the Sato-Levine and Casson invariants}
\label{sec:relationships}

\subsection{Sato-Levine invariant}
\label{sec:satolevine}

Let $\L=L_{1}\cup L_{2}$ denote a 2-component link with linking number zero. Then for $i=1, 2$, component $L_{i}$ bounds a Seifert surface $\Sigma_{i}$ in $B^{4}$ such that $\Sigma_{i}\cap L_{j}=\varnothing$ for $i\neq  j$. Let $L_{12}=\Sigma_{1}\cap \Sigma_{2}$ denote the link with framing induced from $\Sigma_{1}$ (or $\Sigma_{2}$). The self-intersection number of $L_{12}$ is called the \emph{Sato-Levine invariant} $\beta(\L)$, due to Sato \cite{Sato} and independently Levine (unpublished). 

The Conway polynomial of $\L$ of $n$ components is
\[
	\nabla_\L (z) = z^{n-1}(a_0 + a_2z^2 + a_4z^4 + \cdots ), \qquad a_i\in\Z.
\]
We will write $a_i(\L) = a_i$ when we want to emphasize the link. For a link $\L$ of two components, we normalize the Conway polynomial so that 
\[
	\nabla_\L(t^{1/2}-t^{-1/2})=-(t^{1/2}-t^{-1/2})\Delta_{\L}(t, t),
\]
where $\Delta_{\L}(t_1, t_2)$ denotes the multi-variable Alexander polynomial of $\L$. 
The first coefficient $a_0$ is $-lk(L_1, L_2)$ by \cite{Hoste:Conway}. When $a_{0}=0$, write $\tilde{\nabla}_\L(z) = \nabla_\L(z)/z^3$. Then $\tilde{\nabla}_\L(0) = a_{2}=-\beta(\L)$ by \cite{Sturm}. 

Since $lk(L_1, L_2)=0$, the Torres conditions \cite{Torres},
\[
\Delta_{\L}(t_1, 1)=\dfrac{1-t_1^{lk(L_{1}, L_{2})}}{1-t_1} \Delta_{L_{1}}(t_1), \qquad
\Delta_{\L}(1, t_2)=\dfrac{1-t_2^{lk(L_{1}, L_{2})}}{1-t_2} \Delta_{L_{1}}(t_2),
\]
imply that $\Delta_{\L}(t_1, 1)=0$ and $\Delta_{\L}(1, t_2)=0$. Hence, we can write
\[
\Delta_{\L}(t_1, t_2) = t_1^{-1/2}t_2^{-1/2}(t_1-1)(t_2-1)\tilde{\Delta}'_\L(t_1, t_2),
\]
where $\Delta_\L$ is normalized as in equation \eqref{mva}.
\begin{lemma}
Let $\L=L_{1}\cup L_{2}$ be a link with linking number zero. Then 
\[
\beta(\L)=\tilde{\Delta}'_\L(1, 1).
\]

\end{lemma}

\begin{proof}
After setting $t_1=t_2=t$ to obtain the single variable Alexander polynomial, we have
\[
\Delta_\L(t, t) = (t^{1/2} - t^{-1/2})^2 \tilde{\Delta}'_\L(t, t) = -z^2 \tilde{\nabla}_\L(z)
\]
where the last equality is with the change of variable $z=t^{1/2} - t^{-1/2}$. Setting $t=1$ we obtain $\tilde{\Delta}'_\L(1, 1) = -\tilde{\nabla}_\L(0) =\beta(\L)$.
\end{proof}

\begin{lemma}
Let $\L=L_{1}\cup L_{2}$ be an L-space link with linking number zero. Then $\beta=-\sum_{s_1,s_2} h'(s_1, s_2)$ where $h'(s_1,s_2)=h(s_1,s_2)-h_1(s_1)-h_2(s_2)$.
\end{lemma}
Note that by stabilization (Lemma \ref{lem: h increases}) and Lemma \ref{h-function bdy}, $h'(s_1, s_2)$ has finite support, so the above sum makes sense.
\begin{proof}
Since
\[
\tilde{\Delta}'_\L(t_{1}, t_{2})=\sum q_{s_1, s_2}t_{1}^{s_1}t_{2}^{s_2},
\]
and
\[
\tilde{\Delta}_\L(t_{1}, t_{2})=(t_{1}-1)(t_{2}-1)\tilde{\Delta}'_\L(t_{1}, t_{2})=\sum a_{s_1, s_2}t_{1}^{s_1}t_{2}^{s_2},\]
the coefficients are related by
\[
a_{s_1, s_2}=q_{s_1, s_2}-q_{s_{i}-1, s_2}-q_{s_1, s_{2}-1}+q_{s_{1}-1, s_{2}-1}.
\]
Recall that the inclusion-exclusion formula \eqref{computation of h-function 1} gives the coefficients of the Alexander polynomial in terms of the $h$-function of $\L$ as
\begin{multline*}
a_{s_1, s_2}=\chi(HFL^{-}(\L, (s_1, s_2)))=\\
-H(s_1, s_2)+H(s_1-1, s_2)+H(s_1, s_2-1)-H(s_1-1, s_2-1).
\end{multline*}
Observe that $h'(s_1, s_2)$, as defined above, can also be written
\[
	h'(s_1, s_2)=H(s_1, s_2)-H_{1}(s_1)-H_{2}(s_2)
\] 
where $H_1$ and $H_2$ denote the $H$-function of $L_1$ and $L_2$, respectively. Then 
\begin{eqnarray*}
 a_{s_{1}, s_{2}} &=& -h'(s_1, s_2)+h'(s_1-1, s_2)+h'(s_1, s_2-1)-h'(s_1-1, s_2-1)\\
 &=& q_{s_1,s_2}-q_{s_{1}-1, s_{2}}-q_{s_{1}, s_{2}-1}+q_{s_{1}-1, s_{2}-1}.
\end{eqnarray*}
Note that when $L_1$ and $L_2$ are both unknots, $h'(s_1, s_2)=h(s_1, s_2)$. 

Observe that 
$q_{s_1, s_2}=0$ as $s_1\rightarrow \pm \infty$ and $s_2\rightarrow \pm \infty$, and $h'(s_1, s_2)=0$ as $s_1\rightarrow \pm \infty$ and $s_2\rightarrow \pm \infty$. Therefore,
\[ q_{s_1, s_2}= -h'(s_1, s_2). \]
Hence,
\begin{equation}
\label{eqn:SLh}
	\beta(\L)=\tilde{\Delta}'_\L(1, 1)=\sum q_{s_1, s_2}=-\sum h'(s_1, s_2). \qedhere
\end{equation}
\end{proof}

\begin{remark}
\label{knot coeff}
Similarly, for a knot we have that $a_2=\sum_{s}h(s)$, where $a_2$ is the second coefficient of the Conway polynomial.
\end{remark}

\begin{corollary}
\label{beta nonnegative}
If $\L=L_{1}\cup L_{2}$ is an L--space link with vanishing linking number and $L_i$ are unknots for all $i=1,2$, then $\beta(\L)\le 0$ and $\beta(\L)=0$ if and only if $\L$ is an unlink.
\end{corollary}

\begin{proof}
Since $L_i$ are unknots, we have $h'(i,j)=h(i,j)$ for all $i, j$. 
By Corollary \ref{h nonnegative}, $\beta(\L)=-\sum_{i,j} h(i, j)\leq 0$. If  $\beta(\L)=0$ then $h(i,j)=0$ for all $(i,j)\in \Z^{2}$. Since $\L$ is an L--space link,  $\L$ is an unlink \cite{Liu}. \end{proof}

A link $\L$ is called a \emph{boundary link} if its components $L_1$ and $L_2$ bound disjoint Seifert surfaces in $S^3$.

\begin{corollary}
If $\L=L_{1}\cup L_{2}$ is an L--space link with vanishing linking number and $L_i$ are unknots for all $i=1, 2$, then $\L$ is concordant to a boundary link if and only if $\L$ is an unlink. 
\end{corollary}

\begin{proof}
Clearly the unlink is a boundary link, so instead assume that $\L$ is concordant to a boundary link. For boundary links $\beta$ vanishes by definition. Since $\beta$ is a concordance invariant \cite{Sato}, we get $\beta(\L)=0$. By Corollary \ref{beta nonnegative} we have that $\L$ is an unlink.
\end{proof}


\subsection{Casson invariant}
\label{sec:casson}

Here we assume that $\L=L_{1}\cup L_{2}\cdots \cup L_{n}$ is an oriented link in an integer homology sphere $Y$ with all pairwise linking numbers equal zero, and with framing $1/q_i$ on component $L_i$, for $q_i\in\Z$. Hoste \cite{Hoste:Casson} proved that the Casson invariant $\lambda$ of the integer homology sphere $Y_{1/q_1, \cdots, 1/q_n}(\L)$ satisfies a state sum formula,
\begin{equation}
\label{statesum}
	\lambda(Y_{1/q_1, \cdots, 1/q_n}(\L) ) = \lambda(Y) + \sum_{\L'\subset\L}\left( \prod_{i\in\L'}q_i \right) a_2(\L';Y), 
\end{equation}
where the sum is taken over all sublinks $\L'$ of $\L$. 
For example, given a two-component link $\L = L_1\cup L_2$ in $S^3$ with framings $p_i = +1$, formula \eqref{statesum} simplifies to
\begin{eqnarray}
\label{link casson}
	\lambda(S^3_{p_1, p_2}(\L) ) = - \beta(\L) + a_2(L_1) + a_2(L_2).
\end{eqnarray}
By \os~\cite[Theorem 1.3]{OS:Absolutely}, the Casson invariant agrees with the renormalized Euler characteristic of $HF^+(Y)$,
\begin{equation*}
\label{cassonhf+}
	\lambda(Y) = \chi(HF^+_{red}(Y)) - \frac{1}{2}d(Y),
\end{equation*}
where we omit the notation for the unique $\spinc$-structure.
In terms of the renormalized Euler characteristic for $HF^-(Y)$, we have
\begin{equation*}
\label{cassonhf-}
	\lambda(Y) = -\chi(HF^-_{red}(Y)) - \frac{1}{2}d(Y).
\end{equation*}
where the change in sign is due to the long exact sequence $HF_i^-(Y)\rightarrow HF_i^\infty(Y)\rightarrow HF_i^+(Y)\rightarrow HF_{i-1}^-(Y)$. As in \cite[Lemma 5.2]{OS:Absolutely}, the renormalized Euler characteristic can also be calculated using the finite complex
\begin{equation}
\label{renorm-trunc}
	\lambda(Y) =  -\chi ( HF^-(Y_{gr> -2N-1})) + N+1,
\end{equation}
which has been truncated below some grading $-2N-1$ for $N>>0$.
This can be observed by writing 
\begin{equation}
\label{trunc}
	\chi ( HF^-(Y_{gr> -2N-1})) = \chi(\F[U] / U^{k+1}) +  \chi(HF^-_{red}(Y)),
\end{equation}
where $k = \frac{1}{2}d(Y) + N$, and noting that $d(Y)$ is even because $Y$ is an integer homology sphere. 

\begin{remark}
In \cite{OS:Absolutely} Ozsv\'ath and Szab\'o use the renormalized Euler characteristic for $HF^+$ instead of $HF^-$.
From the long exact sequence one sees that these two Euler characteristics add up to the renormalized Euler characteristic of $HF^{\infty}$ (truncated both at sufficiently large positive and negative degrees), which vanishes. 
This explains the sign change between  \eqref{renorm-trunc} and \cite[Lemma 5.2]{OS:Absolutely}.
\end{remark}

\subsection{The Casson invariant from the $h$-function for knots.}
We will review how to obtain Casson invariant from the $H$-function for $Y=S^3_{\pm1}(K)$ using the mapping cone.
\begin{lemma}
\label{casson knot}
Consider $\pm1$ surgery along a knot $K$ in $S^3$. Then
\[
 \lambda(S^3_{\pm1}(K)) = \sum_{s} \pm h(s)  \mp \sum_s \chi ({\A^0_s})_{tor},
\]
where $(\A^0_s)_{tor}$ denotes the torsion summand of $\A^0_s$ and its Euler characteristic is taken with respect to internal degree. In particular, when $K$ is an L--space knot, $\lambda(S^3_{\pm1}(K)) = \sum_{s} \pm h(s)$.
\end{lemma}
\begin{proof}
Apply observation \eqref{renorm-trunc} to the truncated cone complex $(\C_b, D)$, as defined in Section \ref{subsec:knots surgery}. This complex has now been truncated in two directions: it is truncated so that $-b\leq s \leq b$, for $s\in\Z \cong \relspinc(Y, K)$, and is truncated in every summand so that $gr(x)\geq -2N-1$, $N>>0$ for all chains $x\in\C_b$. Recall from section \ref{subsec:gradings} that the homological degree on the surgery complex is a sum of the internal degree and the cube degree. In particular, each of the summands $\A^0_s$ (in cube degree 1) and $\A^1_s$ (in cube degree 0) has an internal degree $\deg$ that is itself a sum of the Maslov grading and a shift by $\deg z_1(s)$ which does not depend on the knot. 
By Proposition \ref{gradingchange}, $\deg z_0(s) = \deg z_1(s) - 2H(s)$. Combining with equation \eqref{trunc} we calculate the Euler characteristic with respect to internal degree as
\[
	\chi( \A^0_{s})_{>-2N-1} = N+1+\frac{1}{2}\deg z_1(s)-H(s)+ \chi(\A^{0}_{s})_{tors},
\]
\[
	\chi(\A^1_{s})_{>-2N-1} = N+1+\frac{1}{2}\deg z_1(s).
\]
Let $p=+1$, then 
\[
	 \chi(HF^-(Y_{gr> -2N-1})) =   \sum_{-b \leq s \leq b}(-H(s) + \chi (\A^{0}_{s})_{tor})+N+1+\dfrac{1}{2}\deg z_1(-b).\\
\]
where the last two terms  come from $\A^0_{-b}$. 

By  \eqref{renorm-trunc} we obtain:
\[
	\lambda(S^3_{+1}(K))  = \sum_{-b \leq s \leq b}( H(s) - \chi (\A^{0}_{s})_{tor})-\frac{1}{2}\deg z_1(-b).
\]
By taking $K$ to be the unknot $O$ we similarly obtain
\[
	\lambda(S^3_{+1}(O)) = \sum_{-b \leq s \leq b} H_{O}(s) - \frac{1}{2}\deg z_{1}(-b)
\]
where $H_{O}(s_{i})$ denotes the $H$-function for the unknot. 
Noting that $S^3_{+1}(O) = S^3$ and that $\lambda(S^3)$ vanishes, we have
\[
 \lambda(S^3_{+1}(K)) = \sum_{-b \leq s \leq b} (H(s) -H_O(s)-  \chi (\A^{0}_{s})_{tor}) = \sum_{s} (h(s)  - \chi (\A^{0}_{s})_{tor}).
\]

The case of $(-1)$--surgery is similar, except that in the mapping cone there is one extra $\A^1$ summand and 
$\A^0$ and $\A^1$  switch parity,
 so that we obtain the equation
\[
 \lambda(S^3_{-1}(K)) = \sum_{-b \leq s \leq b} (-H(s) +H_O(s) +  \chi (\A^{0}_{s})_{tor}) = \sum_{s} (-h(s)  +\chi (\A^{0}_{s})_{tor}).
\]
Finally, notice that when $K$ is an L--space knot, $\chi (\A^{0}_{s})_{tor}$ vanishes. We can see that this agrees with the state sum property \eqref{statesum} of the Casson invariant,
\[
	\lambda(S^3_{1/q}(K)) -\lambda(S^3) = q a_2 (K) = \pm \sum_s h(s),
\]
in the special case $q=\pm1$. 
\end{proof}

\subsection{The Casson invariant from the $h$-function for links.}
For a 2-component link $\L=L_{1}\cup L_{2}$ with vanishing linking number, we can now describe the Casson invariant of $(\pm1, \pm1)$-surgery in terms of the $H$-function, and recover equation \eqref{link casson}. 

\begin{proposition}
\label{cassonlink}
Consider $(p_1, p_2)$ surgery along an L--space link $\L=L_1\cup L_2$ of linking number zero when $p_1, p_2=\pm1$. Then
\[
	\lambda(S^3_{p_1, p_2} (\L)) = p_1p_2\sum_{\bm{s}\in \H(\L)} h'(\bm{s})+ p_1\sum_{s_{1}\in \Z} h_{1}(s_{1})+p_2\sum_{s_{2}\in \Z} h_{2}(s_{2}) .
\]
In particular, 
 \[ 
 \lambda(S^3_{p_1, p_2}) = -p_1p_2\beta(\L)+ p_1a_2(L_{1})+ p_2a_2(L_{2}).
 \]
\end{proposition}

\begin{proof}
Assume first that $p_1,p_2>0$.
Consider the truncated complex $(\C_Q(\mathcal{H}^\L, \Lambda), D)$. For each complete circle contained in the square $Q$, we calculate the local Euler characteristic as follows.
\begin{lemma}
For a 2-component L--space link $\L=L_{1}\cup L_{2}$ with vanishing linking number, and 
 $\bm{s}\in \Z^2$, the Euler characteristic of the chain complex 
\[
\mathfrak{D}_{\bm{s}}= 
\xymatrix{
\A^{10}_{\bm{s}} \ar[d]_{\Phi_{\bm{s}}^{L_2}}& \A^{00}_{\bm{s}} \ar[l]^{\Phi_{\bm{s}}^{L_1}}  \ar[d]^{\Phi_{\bm{s}}^{L_2}} \\
\A^{11}_{\bm{s}} & \A^{01}_{\bm{s}} \ar[l]_{\Phi_{\bm{s}}^{L_1}} 
}
\]
equals 
$$
-h'(\bm{s}) =-H(\bm{s})+H_{1}(s_{1})+H_{2}(s_{2})  
$$   
\end{lemma}
\begin{proof}
We can explicitly calculate the Euler characteristic of $\mathfrak{D}_{gr>-2N-1}$, where all chains have been truncated below some grading $-2N-1$ for $N>>0$. By applying \eqref{trunc} and Proposition \ref{gradingchange} we have the following Euler characteristics with respect to internal degree:
\begin{eqnarray*}
	\chi(\A^{00}_{\bm{s}})_{>-2N-1} &=& N +1- H(\bm{s}) +\frac{1}{2} \deg z_{1,2}(\bm{s})\\  
	\chi(\A^{01}_{\bm{s}})_{>-2N-1} &=& N+1 - H_1(s_1) +\frac{1}{2} \deg z_{1,2}(\bm{s})\\ 
	\chi(\A^{10}_{\bm{s}})_{>-2N-1} &=& N+1 - H_2(s_2) +\frac{1}{2} \deg z_{1,2}(\bm{s})\\ 
	\chi(\A^{11}_{\bm{s}})_{>-2N-1} &=& N+1 +\frac{1}{2} \deg z_{1,2}(\bm{s}).
\end{eqnarray*}

By noting the cube grading of $0, 1$, or $2$, we have that $\A^{00}_{\bm{s}}, \A^{11}_{\bm{s}}$ are supported in the even parity, and $\A^{10}_{\bm{s}},\A^{01}_{\bm{s}}$ are supported in the odd parity. Finally, notice that $\chi(\mathfrak{D}_{\bm{s}})$ agrees with the Euler characteristic of the truncated square, which equals 
\[
- H(\bm{s})+ H_1(s_1) + H_2(s_2). \qedhere
\] 
\end{proof}

Similarly, the Euler characteristics of the chain complexes
\[
	\A^{01}_{\bm{s}} \stackrel{\Phi_{\bm{s}}^{L_1}}{\longrightarrow} \A^{11}_{\bm{s}} \text{ and } 
	\A^{10}_{\bm{s}} \stackrel{\Phi_{\bm{s}}^{L_2}}{\longrightarrow} \A^{11}_{\bm{s}}
\]
are equal to $H_1(s_1)$ and $H_2(s_2)$, respectively.

Consider $Y=S^{3}_{p_{1}, p_{2}}(\L)$. If $p_{1}=p_{2}=1$, then we can choose an appropriate truncation $b>0$ such that $h'(\bm{s})=0$ for all $\bm{s}\notin Q$ and $h'(\pm b, \pm b)=0$. The truncated surgery complex $\C_Q$ contains all circles in the square $Q$ except the crosses as shown in Figure \ref{positive}. The chain complex consisting of the crosses inside one circle has Euler characteristic $H_{2}(s_{2})$ or $H_{1}(s_{1})$ depending on whether the circle lies on the vertical boundary or the horizontal boundary of $Q$. Thus the Euler characteristic is
\begin{eqnarray}
\label{lots of terms}
\begin{aligned}
\chi(\C_Q)_{>-2N-1} &= -\sum_{\bm{s}\in Q}h'(\bm{s})-\sum_{-b\le s_{1}\le b} H_{1}(s_{1}) -\sum_{-b\le s_{2}\le b} H_{2}(s_{2}) \\
 & + \chi(\mathfrak{A}^{11}_{(-b, -b)})_{>-2N-1} ,
\end{aligned}
\end{eqnarray}
where the last term handles the circles at the corners of the truncated complex. 
As in the knot case, we apply the relation \eqref{renorm-trunc} between the Casson invariant and renormalized Euler characteristic (which causes a sign change). We then subtract from \eqref{lots of terms} the corresponding formula for the unlink to obtain
\begin{eqnarray*}
\lambda(Y)- \lambda(S^{3}_{1, 1}(O)) &=& \sum_{\bm{s}\in \Z^2} h'(\bm{s})+\sum_{s_{1}\in \Z} h_{1}(s_{1})+\sum_{s_{2}\in \Z} h_{2}(s_{2}).
\end{eqnarray*}

From \eqref{eqn:SLh}  we get 
\[
a_2(\L)= - \beta(\L)=\sum_{\bm{s}\in \Z^2} ( H(\bm{s})-H_{1}(s_{1})-H_{2}(s_{2})).
\]
By Remark \ref{knot coeff},
\[
a_2(L_{i})= \sum_{s_{i}\in \Z}( H_{i}(s_{i})-H_{O}(s_{i}))
\]
for $i=1, 2$ where $H_{O}(s_{i})$ denotes the $H$-function for the unknot. 
Thus 
\[
\lambda(Y)= -\beta(\L)+ a_2(L_{1})+ a_2(L_{2}).
\]
This recovers  \eqref{link casson} for $p_{1}=p_{2}=1$. The argument is similar in the case where $p_1=p_2=-1$ or $p_1p_2=-1$, modulo possible parity shifts. When $p_{1} p_{2}>0$, the homology
of the cone is supported in cube degree two or zero, and when $p_1 p_2=-1$, the homology is supported in cube degree one (corresponding with the three cases of Theorem \ref{square erase}). 
Also, for negative surgery coefficients the erased part of the  boundary of $Q$ would appear with the opposite coefficient.
In general, for $p_1, p_2 = \pm 1$ we recover
\[
\lambda(Y)= - p_1p_2\beta(\L)+ p_1a_2(L_{1})+ p_2a_2(L_{2}).\qedhere
\]
\end{proof}

\begin{corollary}
Let $L=L_1\cup L_2$ be an $L$-space link with vanishing linking number and unknotted components, and let $L'_{2}$ be the knot obtained from $L_2$ after blowing down the $+1$-framed knot $L_1$. Then for the torsion part $\mathfrak{A}_{s}^{0}$ corresponding to the knot $L'_2$, we have 
\[
	\sum_{s\in \Z}\chi(\mathfrak{A}^{0}_{s})_{tor}=-\sum_{\{(s_1, s_2)\in \Z^{2}| s_1\neq 0\}} h_{\L}(s_1, s_2).
\]
\end{corollary}

\begin{proof}
By Proposition \ref{cassonlink} and Lemma \ref{casson knot}, 
\begin{eqnarray*}
\lambda(S^{3}_{1, 1}(\L))=\sum_{\bm{s}\in \Z^{2}} h_{\L}(\bm{s})=\lambda(S^{3}_{1}(L'_{2}))&=&\sum_{s\in \Z}h_{L'_{2}}(s)-\sum_{s\in \Z}\chi(\mathfrak{A}_{s}^{0})_{tor}\\
	&=& \sum_{s_2\in \Z}h(0, s_2)-\sum_{s\in \Z}\chi(\mathfrak{A}_{s}^{0})_{tor}.
\end{eqnarray*}

Hence,
\[
	\sum_{s\in \Z}\chi(\mathfrak{A}^{0}_{s})_{tor}=-\sum_{\{(s_1, s_2)\in \Z^{2}| s_1\neq 0\}} h_{\L}(s_1, s_2). \qedhere
\]
\end{proof}

\begin{remark}
If there exists a lattice point $(s_1, s_2)$ where $s_1\neq 0$ such that $h_{\L}(s_1, s_2)>0$, then $\sum_{s\in \Z}\chi(\mathfrak{A}^{0}_{s})_{tor}<0$ by Corollary \ref{h nonnegative}. Hence $L'_2$ is not an $L$-space knot. This also follows from Corollary \ref{cor: L2prime L space}.
 
\end{remark}

\begin{example}
Let $\Sigma(2, 3, 5)$ denote the Poincar\'e homology sphere, oriented as the boundary of the four-manifold obtained by plumbing the negative-definite $E8$ graph, i.e. the plumbing along the $E8$ Dynkin diagram with vertex weights all $-2$. In the equality
\[
	\lambda(Y) = \chi(HF^+_{red}(Y)) - \frac{1}{2}d(Y),
\]
we must assume that the Casson invariant $\lambda(Y)$ is normalized so that $\lambda(\Sigma(2, 3, 5)) =-1$ (see \cite[Theorem 1.3]{OS:Absolutely}). Therefore $d(\Sigma(2, 3, 5)) =+2$. 
The Poincar\'e homology sphere $\Sigma(2, 3, 5)$ admits an alternate description as $(-1)$-surgery along the left-handed trefoil knot $T(2, -3)$. By reversing orientation, $-\Sigma(2, 3, 5)$ is $(+1)$-surgery along $T(2, 3)$, with $d(\Sigma(2, 3, 5)) =-2$. Now we may observe that
\[
	\lambda(S^3_{+1}(T(2, 3))) = +1 = h(T(2, 3), 0).
\]
\end{example}

\begin{example}
Consider $(+1, +1)$-surgery along the positively-clasped Whitehead link $\L$. Surgery along one component yields a right-handed trefoil in $S^3$, and then $(+1)$-surgery along the remaining component again produces $-\Sigma(2, 3, 5)$. We observe that
\[
	\lambda(S^3_{+1, +1}(\L)) = +1 = -\beta(\L)+ a_2(L_{1})+ a_2(L_{2}) =-(-1)+0+0  = h(\L, (0,0)).
\]
Similarly, consider $(-1, -1)$-surgery along the Whitehead link. Surgery along the first component now yields a figure eight knot in $S^3$, and $(-1)$-surgery along the figure eight knot produces the (oppositely oriented) Brieskorn sphere $-\Sigma(2, 3, 7)$, for which $\lambda(S^3_{-1, -1}(\L)) = +1$. These two cases correspond with homology supported in cube gradings two and zero, respectively, for which there is no parity change in the Euler characteristic calculation. 

Alternatively, consider $(+1, -1)$ or $(-1, +1)$-surgery along the Whitehead link. This is the (positively oriented) Brieskorn sphere $\Sigma(2, 3, 7)$. It has homology supported in cube grading one, which induces the sign change yielding $\lambda(S^3_{+1, -1}(\L)) = -1$.
\end{example}


\section{Crossing changes}
\label{sec:skein}

We now extend the skein inequality of Peters \cite[Theorem 1.4]{Peters} to the case of links with pairwise linking number zero. 
We continue to omit the unique $\spinc$-structure on an integer homology sphere from the notation.

\begin{lemma}
\label{borbound}
Let $K\subset Y$ be a genus one knot in  an integral homology three-sphere.  Then we have the following inequalities: 
$$d(Y)-2\leq d(Y_{1}(K))\leq d(Y). $$
\end{lemma}

\begin{proof}
The part $d(Y_{1}(K))\leq d(Y)$ follows from \cite[Corollary 9.14]{OS:Absolutely}. Now we prove the inequality that $d(Y)-2\leq d(Y_{1}(K))$. Since $K$ is a genus one knot, $+1$-surgery is a large surgery, i.e. $HF^{-}(Y_{1}(K))\cong H_{\ast}(A^{-}_{0}(K))$ \cite{OS:Hol}. This is a direct sum of one copy of $\F[U]$ and some $U$-torsion. Define $H_{K}(s)$ by saying that $-2H_{K}(s)$ is the maximal homological degree of the free part of $H_{\ast}(A^{-}_{s}(K))$ for $s\in \Z$, which is the same as Definition \ref{Hfunction}. Then $d(S^{3}_{1}(K))=-2H_{K}(0)$. Note that $H_{K}(0)\leq H_{K}(1)+1$ (the monotonicity of $H_{k}$ holds in an arbitrarily homology sphere and the proof is similar to the one in Proposition \ref{h-function increase}), and  $-2H_{K}(1)=d(Y)$. So $d(Y_{1}(K))\geq d(Y)-2$. 
\end{proof}

\begin{theorem}
\label{thm:bbnd}
Let $\L=L_1\cup\cdots\cup L_n$ be a link of pairwise linking number zero. Given a diagram of $\L$ with a distinguished crossing $c$ on component $L_i$, let $D_+$ and $D_-$ denote the result of switching $c$ to positive and negative crossings, respectively. Then 
\[
	d(S^{3}_{1, \cdots, 1}(D_-)) - 2 \leq d(S^{3}_{1, \cdots, 1}(D_+)) \leq d(S^{3}_{1, \cdots, 1}(D_-)).
\]
\end{theorem}

\begin{proof}
Consider the distinguished crossing $c$ along component $L_i$. 
Let $L_{n+1}$ denote the boundary of a crossing disk, i.e. a small disk at $c$ that intersects $L_i$ geometrically twice and algebraically zero times, as in Figure \ref{crossingchange}. 
The crossing change taking $D_-$ to $D_+$ is accomplished by performing $(+1)$-framed surgery along $L_{n+1}\subset S^{3}_{1, \cdots, 1}(D_{-})$. Let $Y=S^{3}_{1, \cdots, 1}(D_{-})$. It is an integral homology sphere, and $S^{3}_{1, \cdots, 1}(D_{+})=Y_{1}(L_{n+1})$. We claim that the Seifert genus of $L_{n+1}$ in $Y$ is at most 1. One can easily create a genus one surface bounded by $L_{n+1}$ in $Y$, simply by adding a tube in $S^{3}\setminus \L$ along the $L_{i}$ to the crossing disk bounded by $L_{n+1}$ at crossing $c$. Then the inequalities follows from Lemma \ref{borbound}. 
\end{proof}

\begin{figure}[H]
    \centering
    \begin{tikzpicture}
    \node[anchor=south west,inner sep=0] at (0,0) {\includegraphics[width=0.6\textwidth]{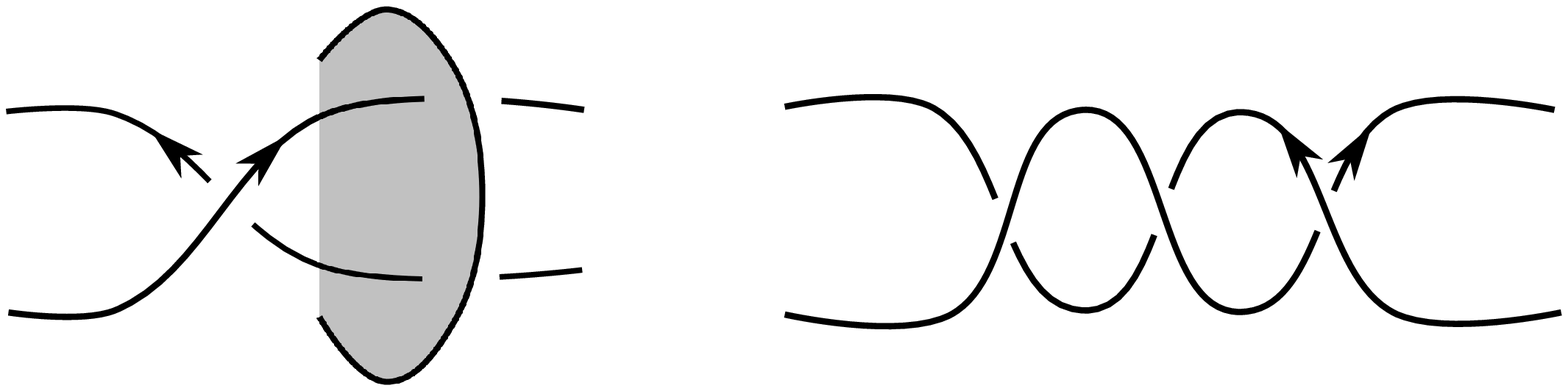}};
    \node[label=above right:{$-1$}] at (2.4, -0.3){};
    \end{tikzpicture}
	\caption{A crossing change taking $D_+$ to $D_-$.}
	\label{crossingchange}
\end{figure}

\section{Genus bounds}
\label{sec:genusbounds}

\subsection{Inequalities}
\label{subsec:genusboundsformulas}

Now we may generalize Peters' and Rasmussen's $4$-ball genus bounds to links with vanishing linking numbers \cite{Peters, Ras}.

Recall that the $n$ components of the link $\L=L_{1}\cup \cdots \cup L_{n}$ bound $n$ mutually disjoint, smoothly embedded surfaces in the $4$-ball if and only if each pairwise linking number is zero. In this case, we define the $4$-genus of $\L$ as:
\[
	g_{4}(\L)=\min \left\{ \sum_{i=1}^{n} g_{i}\mid g_{i}=g(\Sigma_{i}), \Sigma_{1}\sqcup \cdots \sqcup \Sigma_{n}\hookrightarrow B^{4}, \partial \Sigma_{i}=L_{i} \right\},
\]
where the component $L_i$ bounds a surface $\Sigma_i$ with smooth $4$-genus $g_i$.

Let $B_{p_i}$ denote a circle bundle over a closed oriented genus $g_i$ surface with Euler characteristic $p_i$. We have that $H^{2}(B_{p_i})\cong \Z^{2g_i}\oplus \Z_{p_i}$ (see for example \cite[Proposition 3.1]{Liu} for a homology calculation). In \cite{Liu}, the second author constructed a $\spinc$-cobordism from $(\Co, \mft')$ to $(\Sp, \mft)$. Following our conventions for the parameterization of $\spinc$-structures (section \ref{subsec:spinc}), the labelling of the torsion Spin$^{c}$-structures $\mft_{i}$ on $B_{p_i}$ is such that $-|p_i|/2\leq t_i \leq |p_i|/2$,  corresponding to the torsion part of $H^{2}(B_{p_{i}})$. 

We are ready to prove Proposition   \ref{prop:genusbound}. We restate it here for the reader's convenience.

\begin{proposition}
Let $\L\subset S^{3}$ denote an $n$-component link with pairwise vanishing linking numbers. Assume that $p_{i}>0$ for all $1\leq i \leq n$. Then 
\begin{equation}
\label{first}
 d(S^{3}_{-p_1, \cdots, -p_n}(\L), \mft)\leq \sum_{i=1}^{n} d(L(-p_i, 1), t_i) +2f_{g_{i}}(t_i)
 \end{equation}
 and
 \begin{equation}
 \label{second}
-d(S^{3}_{p_1, \cdots, p_n}(\L), \mft)\leq \sum_{i=1}^{n} d(L(-p_i, 1), t_i) +2f_{g_{i}}(t_i).
 \end{equation}
\end{proposition}

\begin{proof}

By \cite[Proposition 3.8]{Liu} we get the inequality
\begin{equation}
\label{dbound}
	d(S^{3}_{-p_{1}, \cdots, -p_{n}}(\L), \mft)\leq \sum\limits_{i=1}^{n} d_{bot}(B_{-p_{i}}, t_{i})+g_{1}+\cdots +g_{n}.
\end{equation}
By \eqref{d circle bundle via f} we can rewrite the right hand side as
\[
\sum\limits_{i=1}^{n} d_{bot}(B_{-p_{i}}, t_{i})+g_{1}+\cdots +g_{n}=\sum_{i=1}^{n}(-\phi(p_i,t_i)+2f_{g_i}(t_i)).
\] 
This proves the first inequality \eqref{first}. If $\L^{\ast}$ is the mirror of $\L$, then
\[
	d(S^{3}_{\bm{p}}(\L), \mft)=-d(S^{3}_{-\bm{p}}(\L^{\ast}), \mft).
\]
Since mirroring preserves the 4-genera of knots, the right hand side of \eqref{dbound} does not change if we replace $d(S^{3}_{\bm{p}}(\L), \mft)$ by $-d(S^{3}_{-\bm{p}}(\L^{\ast}), \mft)$.
This proves the second inequality \eqref{second}.
\end{proof}

Proposition \ref{prop:genusbound} gives lower bounds on the 4-genera of $\L$ in terms of the 3-manifolds $S^{3}_{\pm \p}(\L)$ where $\p\succ \bm{0}$. Theorem \ref{thm:generalizedniwu} allows us to compute the $d$-invariants of $S^{3}_{\pm \p}(\L)$ for two-component L--space links. Combining these two observations, we obtain the following bounds for the 4-genera of two-component L--space links with vanishing linking number. 

\begin{theorem}

Let $\L=L_{1}\cup L_{2}$ denote a two-component L--space link with vanishing linking number. Then  for all $p_1>0$ and $p_2>0$
\begin{equation*}
h(s_1, s_2) \leq f_{g_{1}}(t_{1})+f_{g_{2}}(t_{2}), 
\end{equation*}
where $(s_1,s_2)\in \Z^2$ corresponds to the $\spinc$-structure $\mft=(t_1,t_2) $ on $S^{3}_{p_{1}, p_{2}}(\L)$.
\end{theorem}

\begin{proof}
By Theorem \ref{thm:generalizedniwu} we have 
$$
-d(S^{3}_{p_1, p_2}(\L), \mft)=-\sum_{i=1}^{2} \phi(p_i,t_i)+2\max \{ h(s_{\pm \pm}(t_1, t_2)).
$$
Combining this with \eqref{second} and dividing by 2, we get
$$\max \{ h(s_{\pm \pm}(t_1, t_2)) \}\leq f_{g_{1}}(t_{1})+f_{g_{2}}(t_{2}).$$
By Lemma \ref{lem: h increases}, $h(s_1, s_2)\le \max  \{ h(s_{\pm \pm}(t_1, t_2)) \}$. Hence
\[
	h(s_1, s_2)\leq f_{g_{1}}(t_{1})+f_{g_{2}}(t_{2}). \qedhere
\]
\end{proof}

\subsection{Examples}
\label{subsec:genusboundsexamples}
There exist some links $\L$ for which the $d$-invariants of the $(\pm 1, \cdots, \pm 1)$-surgery manifolds are known. In this section we provide some examples where existing $d$-invariants calculations can now be applied to determine the 4-genera for several families of links.

\begin{example}
The two bridge link $\L_k=b(4k^{2}+4k, -2k-1)$ is a two-component L--space link with vanishing linking number for any positive integer $k$ \cite{LiuY2}. Theorem \ref{thm:generalizedniwu} implies 
\[ d(S^{3}_{-1, -1}(\L))=0 \]
and
\[ d(S^{3}_{1, 1}(\L))=-2h(0,0)=-2\lceil k/2 \rceil,\]
where the $h$-function of $\L$ can be obtained from the calculation in \cite[Proposition 6.12]{LiuY2}.
When $p_1, p_2$ are sufficiently large positive integers, we obtain that $g_{4}(\L)\geq k$. We may construct two disjoint surfaces bounded by $\L$ such that $g_{4}(\L)=k$. For details, see \cite[Example 4.1]{Liu}. 
\end{example}

Consider the special case of Inequality \eqref{d-gen-ineq}  when $p_1=\cdots=p_n=1$. There is a unique $\spinc$ structure $\mft_0$ on $S^{3}_{\pm 1, \cdots, \pm 1}(\L)$, and we have
\begin{equation}
\label{inequlity 5}
-d(S^{3}_{1, \cdots, 1}(\L),\mft_0)/2 \leq \sum_{i=1}^{n} \lceil g_i/2 \rceil. 
\end{equation}

On the one hand, this inequality can be used to restrict the $d$-invariants of $(\pm 1)$-surgery along a genus one link $\L$ with vanishing pairwise linking numbers. This will be the case in Corollary \ref{genus1knot}. On the other hand, we may bound the 4-genus of a link $\L$ if we know $d(S^{3}_{1, \cdots, 1}(\L))$. This will be the case in Example \ref{2bwhitehead}.

\begin{corollary}
\label{genus1knot}
Let $\L$ denote a genus one link with vanishing pairwise linking numbers. Then $d(S^{3}_{1, \cdots, 1}(\L), \mft_{0})=0$ or $-2$, and $d(S^{3}_{-1, \cdots, -1}(\L), \mft_{0})=0$ or $2$. 
\end{corollary}

\begin{proof}
By inequality \eqref{inequlity 5}, 
\[ d(S^{3}_{1, \cdots, 1}(\L), \mft_0)\geq -2.\]
By observing the negative definite cobordism from $S^3_{1, \cdots, 1}(\L)$ to $S^3$, we have\newline $d(S^{3}_{1, \cdots, 1}(\L), \mft_0)\leq 0$. Note also that $d(S^{3}_{1, \cdots, 1}(\L), \mft_0)$ is even because $S^{3}_{1, \cdots, 1}(\L)$ is an integer homology sphere. Then $d(S^{3}_{1, \cdots, 1}(\L), \mft_0)=0$ or $-2$. 

Let $\L^{\ast}$ denote the mirror link of $\L$. Then 
$d(S^{3}_{-1, \cdots, -1}(\L), \mft_0)=-d(S^{3}_{1, \cdots, 1}(\L^{\ast}), \mft_0)$ equals $0$ or $2$ since $\L^{\ast}$ is also a genus one link. 
\end{proof}

Let $D_{+}(K, n)$ denote the $n$-twisted positively clasped Whitehead double of $K$.  
If $K$ is an unknot, then $D_{+}(K, n)$ is also an unknot. Otherwise, $D_{+}(K, n)$ is a genus one knot. Corollary \ref{genus1knot} tells us that $d(S^{3}_{1}(D_{+}(K, n)))=0$ or $-2$ and $d(S^{3}_{-1}(D_{+}(K, n)))=0$ or $2$. Indeed, using Hedden's calculation of $\tau(K)$ for Whitehead doubles \cite{Hedden}, Tange calcuated $HF^{+}(S^{3}_{\pm 1}(D_{+}(K, n)))$ for any knot $K$, yielding:
\begin{proposition}\cite{Motoo}
Let $K$ be a knot in $S^{3}$. Then 
\[
d(S^{3}_{1}(D_{+}(K, n)), \mft_0) = \left\{
        \begin{array}{ll}
            0  & \quad n\geq 2\tau(K) \\
            -2 & \quad n<2\tau(K)
        \end{array}
    \right. 
\]
and 
\[ d(S^{3}_{-1}(D_{+}(K, n)), \mft_0)=0. \]
\end{proposition}
This calculation restates Hedden's criterion on the sliceness of $D_+(K, n)$ in terms of the $d$-invariant: if $n<2\tau(K)$, then $D_+(K, n)$ is not slice. 

\begin{example}
\label{bingdouble}
Let $B(K)$ be an untwisted Bing double of $K$. We label the component involving $K$ as $L_2$ and the other unknotted component as $L_1$. Then 
\[
	d(S^{3}_{1, 1}(B(K), \mft_0)=d(S^{3}_{1}(D_{+}(K, 0)), \mft_0).
\]

Since $B(K)$ is related to $D_+(K, 0)$ by a band move, when $B(K)$ is slice, this implies $D_+(K, 0)$ is slice. In particular, whenever $\tau(K)>0$, then $B(K)$ is not slice. A genera-minimizing pair of surfaces may be constructed as follows. Since both components $L_1$ and $L_2$ are unknots, they bound disks which intersect transversely at two points in $B^{4}$. Add a tube to cancel this pair of intersection points and increase the total genus by one. This illustrates that the bound given by Inequality \ref{d-gen-ineq} is sharp, since
\[
	2=-d(S^{3}_{1, 1}(B(K), \mft_0)=-d(S^{3}_{1}(D_{+}(K, 0)), \mft_0)\leq 2\lceil g_1/2\rceil+ 2\lceil g_2/2\rceil
\]
implies that $g_1+g_2 \geq 1$.
\end{example}

\begin{example}
\label{2bwhitehead}
Let $W$ denote the Whitehead link and $\L$ denote the 2-bridge link $b(8k ,4k+1)$ where $k\in \mathbb{N}$. By the work of Y. Liu \cite[Theorem 6.10]{LiuY1}, 
\[
	HF^{-}(S^{3}_{\pm 1, \pm 1}(\L))\cong HF^{-}(S^{3}_{\pm 1, \pm 1}(W))\oplus \F^{k-1}.
\]
Then the $d$-invariant $d(S^{3}_{(\pm 1, \pm 1)}(\L))$ is the same as the one for the Whitehead link. Hence  by \cite[Proposition 6.9]{LiuY1}, 
\[
	d(S^{3}_{1, 1}(\L), \mft_0)=d(S^{3}_{1, 1}(W), \mft_0)=-2.
\]
By Inequality \ref{inequlity 5}, we have 
\[
	\lceil g_1/2 \rceil +\lceil g_2/2 \rceil \geq 1.
\]
Observe that both the link components of $\L$ are unknots. Again we add a tube to eliminate the intersection, obtaining pairwise disjoint surfaces with total genus one. Hence $g_4(\L)= 1$, and the bound obtained by Inequality \ref{d-gen-ineq}  is sharp.
\end{example}

\subsection*{Acknowledgements}

We would like to thank Yajing Liu, Charles Livingston, Kyungbae Park, Mark Powell and Jacob Rasmussen for the useful discussions. We deeply appreciate the anonymous referee for a thorough and very helpful review and numerous suggestions. In particular, the simplified proof of Theorem \ref{thm:bbnd}  was suggested by the referee.
The work of E. G. and B. L. was partially supported by the NSF grant DMS-1700814. The work of A. M.  was partially supported by the NSF grant DMS-1716987.
E. G. was also supported by Russian Academic Excellence Project 5-100 and the grant RFBR-16-01-00409.

\bibliographystyle{alpha}
\bibliography{biblio}

\end{document}